\documentclass[11pt, a4paper, pagebackref, oneside, final]{amsart}
\usepackage[notref,notcite]{showkeys}
\usepackage[latin1]{inputenc}
\usepackage[T1]{fontenc}
\usepackage[stretch=10]{microtype}
\usepackage{fixltx2e}
\usepackage{fix-cm}
\usepackage{hyphenat}
\usepackage{dsfont}
\usepackage[DIV11]{typearea}
\usepackage{hyperref}

\makeatletter
\providecommand{\SG@adddot}{.}
\def\tospace#1{\@tospace#1 \tospace@delimiter}
\def\@tospace#1 #2\tospace@delimiter{#1}
\def\MR#1{\edef\MR@help{{http://www.ams.org/mathscinet-getitem?mr=\tospace{#1}}{\tospace{#1}}}%
\expandafter\href\MR@help\SG@adddot}
\makeatother
\renewcommand*{\backref}[1]{}
\renewcommand*{\backrefalt}[4]{%
	\ifcase #1 %
	\or
	  Cited page #2.
	\else
	  Cited pages #2.
	\fi
}

\newcommand{\Proba}{\mathbb{P}}
\newcommand{\E}{\mathbb{E}}
\newcommand{\N}{\mathbb{N}}
\newcommand{\R}{\mathbb{R}}

\newcommand{\dd}{\mathop{}\!\mathrm{d}}
\newcommand{\norm}[1]{\left\| #1 \right\|}
\newcommand{\st}{\,:\,}
\newcommand{\boF}{\mathcal{F}}
\newcommand{\boE}{\mathcal{E}}
\newcommand{\boB}{\mathcal{B}}
\newcommand{\boC}{\mathcal{C}}
\newcommand{\hatT}{\mathcal{L}}
\newcommand{\Z}{\mathbb{Z}}
\newcommand{\breakdots}{,\allowbreak\dotsc,\allowbreak}
\newcommand{\ic}{\mathbf{i}}

\DeclareMathOperator{\Lip}{Lip}

\DeclareMathOperator{\Card}{Card}
\newcommand{\coloneqq}{\mathrel{\mathop:}=}
\newcommand{\eqqcolon}{=\mathrel{\mathop:}}

\newtheorem{thm}{Theorem}[section]
\newtheorem{prop}[thm]{Proposition}
\newtheorem{definition}[thm]{Definition}
\newtheorem{lem}[thm]{Lemma}
\newtheorem{cor}[thm]{Corollary}
\newtheorem*{prop*}{Proposition}

\theoremstyle{definition}

\newtheorem{rmk}[thm]{Remark}

\numberwithin{equation}{section}

\title{Optimal concentration inequalities for dynamical systems}
\author{Jean-Ren\'e Chazottes, S\'ebastien Gou\"ezel}

\address{CPHT, CNRS UMR 7644, Ecole polytechnique, 91128 Palaiseau Cedex, France}
\email{chazottes@cpht.polytechnique.fr}
\address{IRMAR, CNRS UMR 6625,
Universit\'e de Rennes 1, 35042 Rennes, France}
\email{sebastien.gouezel@univ-rennes1.fr}

\date{May 11, 2012}

\begin{document}

\begin{abstract}
For dynamical systems modeled by a Young tower with exponential
tails, we prove an exponential concentration inequality for all
separately Lipschitz observables of $n$ variables. When tails are
polynomial, we prove polynomial concentration inequalities. Those
inequalities are optimal. We give some applications of such
inequalities to specific systems and specific observables.
\end{abstract}

\maketitle

\setcounter{tocdepth}{1} \tableofcontents

\section{Introduction}

Let $X$ be a metric space. A function $K$ on $X^n$ is separately
Lipschitz if, for all $i$, there exists a constant $\Lip_i(K)$ with
  \begin{equation*}
  |K(x_0,\dotsc, x_{i-1},x_i,x_{i+1},\dotsc, x_{n-1})
  - K(x_0,\dotsc, x_{i-1},x'_i,x_{i+1},\dotsc, x_{n-1})|
  \leq \Lip_i(K) d(x_i, x'_i),
  \end{equation*}
for all points $x_1,\dotsc,x_n, x'_i$ in $X$.

Consider a stationary process $(Z_0,Z_1,\dotsc)$ taking values in
$X$. We say that this process satisfies an exponential concentration
inequality if there exists a constant $C$ such that, for any
separately Lipschitz function $K(x_0,\dotsc, x_{n-1})$, one has
  \begin{equation}\label{eq_exp-concentration}
  \E( e^{K(Z_0,\dotsc, Z_{n-1}) - \E(K(Z_0,\dotsc, Z_{n-1}))}) \leq e^{C\sum_{j=0}^{n-1} \Lip_j(K)^2}.
  \end{equation}
One should stress that this inequality is valid for all $n$ (i.e.,
the constant $C$ does not depend on the number of variables one is
considering).
An important consequence of such an inequality is a control on the
deviation probabilities: for all $t>0$,
  \[
  \Proba\Bigl(|K(Z_0,\dotsc, Z_{n-1}) - \E(K(Z_0,\dotsc, Z_{n-1}))|>t\Big)
  \leq 2 e^{-\frac{t^2}{4C\sum_{j=0}^{n-1} \Lip_j(K)^2}}.
  \]
This inequality follows from the inequality $\Proba(Y>t)\leq
e^{-\lambda t} \E(e^{\lambda Y})$ ($\lambda>0$) with $Y=K(Z_0,\dotsc,
Z_{n-1}) - \E(K(Z_0,\dotsc, Z_{n-1}))$,  then we use
inequality~\eqref{eq_exp-concentration} and optimize over $\lambda$
by taking $\lambda = t/(2C\sum_{j=0}^{n-1} \Lip_j(K)^2)$.

In some cases, it is not reasonable to hope for such an exponential
inequality. One says that $(Z_0,Z_1,\dotsc)$ satisfies a polynomial
concentration inequality with moment $Q\geq 2$ if there exists a
constant $C$ such that, for any separately Lipschitz function
$K(x_0,\dotsc, x_{n-1})$, one has
  \begin{equation}
  \label{eq_LQ_concentration}
  \E \Bigl(|K(Z_0,\dotsc, Z_{n-1}) - \E(K(Z_0,\dotsc, Z_{n-1}))|^Q\Bigr)
  \leq C\left(\sum_{j=0}^{n-1} \Lip_j(K)^2\right)^{Q/2}.
  \end{equation}

An important consequence of such an inequality is a control on the
deviation probabilities: for all $t>0$,
  \begin{equation}
  \label{eq_weak_LQ_concentration}
  \Proba( |K(Z_0,\dotsc, Z_{n-1}) - \E(K(Z_0,\dotsc, Z_{n-1}))| > t)
  \leq C t^{-Q} \left(\sum_{j=0}^{n-1} \Lip_j(K)^2\right)^{Q/2}.
  \end{equation}
The inequality~\eqref{eq_weak_LQ_concentration} readily follows
from~\eqref{eq_LQ_concentration} and the Markov inequality. However,
it is weaker in general. We will say that $(Z_0,Z_1,\dotsc)$
satisfies a weak $L^Q$ concentration inequality
if~\eqref{eq_weak_LQ_concentration} holds for any separately
Lipschitz function $K$.

For instance, if $Z_0,Z_1,\dotsc$ is an i.i.d.~ process, then it
satisfies an exponential concentration inequality if $Z_i$ is bounded
\cite[Page 68]{ledoux_book}, a polynomial concentration inequality
with moment $Q\geq 2$ if $Z_i\in L^Q$ \cite{boucheron_et_al}, and a
weak $L^Q$ concentration inequality if $\Proba(|Z_i|>t) \leq C
t^{-Q}$ (while we could not locate a proper reference in the
literature, this follows easily from classical martingale techniques
and a weak $L^Q$ Rosenthal-Burkholder inequality -- see
Theorem~\ref{thm_rosenthal_weak} below).

\bigskip

Our main goal in this article is to study processes coming from
dynamical systems: we consider a map $T$ on a metric space $X$, and
an invariant probability measure $\mu$. Under suitable assumptions,
we wish to show that the process $(x, Tx,\dotsc)$ (where $x$ is
distributed following $\mu$) satisfies concentration inequalities.
Equivalently, we are interested in the concentration properties of
the measure $\mu_n$ on $X^n$ given by $\dd\mu_n(x_0,\dotsc,
x_{n-1})=\dd\mu(x_0) \delta_{x_1=Tx_0}\dotsm
\delta_{x_{n-1}=Tx_{n-2}}$. This is not a product measure but, if the
map $T$ is sufficiently mixing, one may expect that $T^k(x)$ is more
or less independent of $x$ is $k$ is large, making the process
$(x,Tx,\dotsc)$ look like an independent process to some extent.

Such questions have already been considered in the literature. In
particular, \cite{collet_concentration_BV} proves that a
(non-necessarily Markov) piecewise uniformly expanding map of the
interval satisfies an exponential concentration inequality.
Polynomial concentration inequalities (with moment $2$, also called
Devroye inequalities) have been proved in less expanding situations
(exponential Young towers -- including H\'enon maps -- in
\cite{chazottes_collet_schmitt0}, intermittent map with parameter
close enough to $0$ in \cite{chazottes_intermittent}). Our goal is to
prove optimal concentration inequalities for the same kind of
systems. In particular, we will prove that Young towers with
exponential tails satisfy an exponential concentration inequality,
and that in Young towers with polynomial tails one can get polynomial
concentration with a moment directly related to the tails of the
return time on the basis of the tower.

Concentration inequalities are a tool to bound systematically the
fluctuations of `complicated' observables of the form
$K(x,Tx,\dotsc,T^{n-1}x)$. For instance, the function $K$ can have a
complicated analytic expression or can be implicitly defined (e.g.~as
an optimization problem). If we are able to get a good estimate of
the Lipschitz constants, we can apply the concentration inequality we
have at our disposal. Various examples of observables have been
studied in \cite{collet_concentration_BV,
chazottes_collet_schmitt,chazottes_intermittent}. Since we establish
here optimal concentration inequalities, this improves automatically
the bounds previously available for these observables. We shall state
explicitly some of the new results which can be obtained.

\bigskip

\textbf{Outline of the article:} The proofs we will use for different
classes of systems are all based on classical martingale arguments.
It is enlightening to explain them in the simplest possible
situation, subshifts of finite type endowed with a Gibbs measure. We
will do so in Section~\ref{sec_subshift}. The following 4 sections
are devoted to proofs of concentration inequalities in various kinds
of dynamical systems with a combinatorial nature, namely Young towers
with exponential tails in Section~\ref{sec_exponential}, with
polynomial tails in Section~\ref{sec_non_uniform} (the invertible
case is explained in Section~\ref{sec_non_uniform_invertible}), and
with weak polynomial tails in Section~\ref{sec_non_uniform_weak}.
Several applications to concrete dynamical systems and to specific
observables are described in Section~\ref{sec_applications}. Finally,
an appendix is devoted to the proof of a particularly technical
lemma.

In this paper, the letter $C$ denotes a constant that can vary from
line to line (or even on a single line).

\section{Subshifts of finite type}

\label{sec_subshift}

In this section, we describe a strategy to prove concentration
inequalities. It is very classical, uses martingales, and was for
instance implemented for dynamical systems in
\cite{collet_concentration_BV} and for weakly dependent processes in
\cite{rio_concentration}. Our proofs for more complicated systems
will also rely on this strategy. However, it is enlightening to
explain it in the most simple situation, subshifts of finite type.

\subsection{Unilateral subshifts of finite type}

Let $X\subset \Sigma^\N$ be the state space of a topologically mixing
one-sided subshift of finite type, with an invariant Gibbs measure
$\mu$, and the combinatorial distance $d(x, y)=\beta^{s(x, y)}$ where
$\beta<1$ is some fixed number and $s(x,y)$ is the separation time of
$x$ and $y$, i.e., the minimum number $n$ such that $T^n x$ and $T^n
y$ do not belong to the same element of the Markov partition. Writing
$x=(x_0 x_1\dotsc)$ and $y=(y_0 y_1 \dotsc)$, then $s(x,y)=\inf\{n
\st x_n\not=y_n\}$.

\begin{thm}
\label{thm_subshift}
The system $(X,T,\mu)$ satisfies an exponential concentration
inequality.
\end{thm}

Fix a separately Lipschitz function $K(x_0,\dotsc, x_{n-1})$. We
consider it as a function on $X^\N$ depending only on the first $n$
coordinates (therefore, we will write $\Lip_i(K)=0$ for $i\geq n$).
We endow $X^{\N}$ with the measure $\mu_\infty$ limit of the $\mu_N$
when $N\to\infty$. On $X^\N$, let $\boF_p$ be the $\sigma$-algebra of
events depending only on the coordinates $(x_j)_{j\geq p}$ (this is a
decreasing sequence of $\sigma$-fields). We want to write the
function $K$ as a sum of reverse martingale differences with respect
to this sequence. Therefore, let $K_p = \E(K|\boF_p)$ and
$D_p=K_{p}-K_{p+1}$. The function $D_p$ is $\boF_p$-measurable and
$\E(D_p|\boF_{p+1}) =0$. Moreover, $K-\E(K) = \sum_{p\geq 0} D_p$.

The main point of the proof is to get a good bound on $D_p$:
\begin{lem}
\label{lem_close}
There exist $C>0$ and $\rho < 1$ such that, for any $p$, one has
  \begin{equation*}
  |D_p|\leq C\sum_{j=0}^p \rho^{p-j} \Lip_j(K).
  \end{equation*}
\end{lem}

We then use the Hoeffding-Azuma inequality (see e.g.~\cite[Page
33]{milman_schechtman} or \cite[Page 68]{ledoux_book}), saying that
for such a sum of martingale increments,
  \begin{equation*}
  \E(e^{\sum_{p=0}^{P-1} D_p}) \leq e^{\sum_{p=0}^{P-1} \sup |D_p|^2}.
  \end{equation*}
The Cauchy-Schwarz inequality gives
  \begin{equation*}
  \left( \sum_{j=0}^p \rho^{p-j} \Lip_j(K) \right)^2
  \leq \left( \sum_{j=0}^p \rho^{p-j} \Lip_j(K)^2 \right) \cdot
  \left( \sum_{j=0}^p \rho^{p-j} \right)
  \leq C \sum_{j=0}^p \rho^{p-j} \Lip_j(K)^2.
  \end{equation*}
Summing over $p$, we get $\sum_{p=0}^{P-1} \sup |D_p|^2 \leq C\sum_j
\Lip_j(K)^2$. Using the Hoeffding-Azuma inequality at a fixed index
$P$, and then letting $P$ tend to infinity, we get $\E(e^{\sum D_p})
\leq e^{C\sum \Lip_j(K)^2}$, which is the desired exponential
concentration inequality since $\sum D_p = K-\E(K)$.

It remains to prove Lemma~\ref{lem_close}. Let $g$ denote the inverse
of the jacobian of $T$, and $g^{(k)}$ the inverse of the jacobian of
$T^k$. Let $\hatT$ denote the transfer operator associated to the map
$T$, defined by duality by $\int u\cdot v\circ T\dd\mu = \int \hatT u
\cdot v \dd\mu$. It can be written as $\hatT u(x)=\sum_{Ty=x}g(y)
u(y)$. In the same way, $\hatT^k u(x)=\sum_{T^k y =x} g^{(k)}(y)
u(y)$. One can define a Markov chain by jumping from a point $x$ to
one of its preimages $y$ with the probability $g(y)$, then $\hatT$ is
simply the Markov operator corresponding to this Markov chain. In
particular,
  \begin{align*}
  K_p(x_p, x_{p+1},\dotsc)
  & =\E(K|\boF_p) (x_p,x_{p+1},\dotsc) = \E(K(X_0,\dotsc, X_{p-1}, x_p,\dotsc) |X_p=x_p)
  \\ & = \sum_{T^p(y)=x_p} g^{(p)}(y)K(y,\dotsc, T^{p-1}y, x_p,\dotsc).
  \end{align*}
To prove that $D_p$ is bounded, i.e., $K_p$ is close to $K_{p+1}$,
one should show that this quantity does not depends too much on
$x_p$. The preimages of $x_p$ under $T^p$ equidistribute in the
space, therefore one should be able to show that $K_p$ is close to an
integral quantity. This is done in the following lemma.
\begin{lem}
\label{lem_controle_Kp_fullshift}
We have
  \begin{equation*}
  \left|K_p(x_p,\dotsc) -\int K(y,\dotsc, T^{p-1}y, x_p,\dotsc)\dd \mu(y)\right|
  \leq C \sum_{j=0}^{p-1} \Lip_j(K) \rho^{p-1-j},
  \end{equation*}
where $C>0$ and $\rho<1$ only depend on $(X,T)$.
\end{lem}
This lemma implies in particular that $K_p(x_p, x_{p+1},\dotsc) -
K_p(x'_p,x_{p+1},\dotsc)$ is bounded by $C \sum_{j=0}^{p} \Lip_j(K)
\rho^{p-j}$. Averaging over the preimages $x'_p$ of $x_{p+1}$, we get
the same bound for $D_p(x_p, x_{p+1},\dotsc)$, proving
Lemma~\ref{lem_close}.

\begin{proof}[Proof of Lemma~\ref{lem_controle_Kp_fullshift}]
The equidistribution of the Markov chain starting from $x_p$ is
formulated most conveniently in terms of the transfer operators,
which act on functions of one variable. Therefore, we will eliminate
the variables $x_0,\dotsc, x_{p-1}$ one after the other. Let us fix a
point $x_*$ in $X$, we decompose $K_p$ as
  \begin{align*}
  K_p(x_p,\dotsc)&=\sum_{i=0}^{p-1} \sum_{T^p(y)=x_p} g^{(p)}(y)
  (K(y,\dotsc, T^i y, x_*,\dotsc,x_*, x_p,\dotsc)
  \\
  &\hphantom{=\sum_{i=1}^{p-1} \sum_{T^p(y)=x_p} g^{(p)}(y)(}
  - K(y,\dotsc, T^{i-1}y, x_*,\dotsc, x_*, x_p,\dotsc))\\ &
  \ \ +K(x_*,\dotsc,x_*, x_p,\dotsc).
  \end{align*}
For fixed $i$, we may group together those points $y\in T^{-p}(x_p)$
that have the same image under $T^i$, splitting the sum
$\sum_{T^p(y)=x_p}$ as $\sum_{T^{p-i}(z)=x_p} \sum_{T^i(y)=z}$. Since
the jacobian is multiplicative, one has $g^{(p)}(y) = g^{(i)}(y)
g^{(p-i)}(z)$. Let us define a function
  \begin{equation}
  \label{eq_define_fi}
  \begin{split}
  f_i(z)&=\sum_{T^i y = z}g^{(i)}(y)(K(y,\dotsc, T^i y, x_*,\dotsc,x_*, x_p,\dotsc)
  \\& \hphantom{= \sum_{T^i y = z}g^{(i)}(y)(}
  - K(y,\dotsc, T^{i-1}y, x_*,\dotsc, x_*, x_p,\dotsc))
  \\ & = \sum_{T^i y = z}g^{(i)}(y) H(y,\dotsc, T^i y).
  \end{split}
  \end{equation}
Denoting by $\hatT$ the transfer operator (which satisfies $\hatT^k
f(x) = \sum_{T^k(z)=x} g^{(k)}(z) f(z)$), we obtain
  \begin{equation*}
  K_p(x_p,\dotsc)=\sum_{i=0}^{p-1} \hatT^{p-i} f_i(x_p) + K(x_*,\dotsc,x_*, x_p,\dotsc).
  \end{equation*}

The function $H$ is bounded by $\Lip_i(K)$, hence $|f_i|\leq
C\Lip_i(K)$ (since $\sum_{T^i y =z} g^{(i)}(y)=1$ by invariance of
the measure). To estimate the Lipschitz norm of $f_i$, we write
  \begin{equation}
  \label{wpoiuxcvjl,mlkxwc}
  \begin{split}
  f_i(z)-f_i(z')= {} &\sum (g^{(i)}(y)-g^{(i)}(y')) H(y,\dotsc, T^i y)
  \\& +
  \sum g^{(i)}(y')(H(y,\dotsc, T^i y)-H(y',\dotsc, T^i y')),
  \end{split}
  \end{equation}
where $z$ and $z'$ are two points in the same partition element, and
their respective preimages $y$, $y'$ are paired according to the
cylinder of length $i$ they belong to. A distortion control gives
$|g^{(i)}(y)-g^{(i)}(y')| \leq C g^{(i)}(y) d(z,z')$, hence the first
sum is bounded by $C\Lip_i(K) d(z,z')$. For the second sum,
substituting successively each $T^j y$ with $T^j y'$, we have
  \begin{equation*}
  |H(y,\dotsc, T^i y)-H(y',\dotsc, T^i y')| \leq 2\sum_{j=0}^i \Lip_j(K) d(T^j y, T^j y') \leq
  2\sum_{j=0}^i \Lip_j(K) \beta^{i-j} d(z,z').
  \end{equation*}
Summing over the different preimages of $z$, we deduce that the
Lipschitz norm of $f_i$ is bounded by $C\sum_{j=0}^i \Lip_j(K)
\beta^{i-j}$.

Let $\boC$ be the space of Lipschitz functions on $X$, with its
canonical norm $\norm{f}_\boC=\sup|f|+\Lip(f)$. The operator $\hatT$
has a spectral gap on $\boC$: there exist $C>0$ and $\rho<1$ such
that $\norm{\hatT^k f - \int f \dd\mu}_\boC \leq C \rho^k
\norm{f}_\boC$. We get $\norm{\hatT^{p-i} f_i - \int f_i \dd\mu}_\boC
\leq C \rho^{p-i} \sum_{j=0}^i \Lip_j(K) \beta^{i-j}$. This bound in
$\boC$ implies in particular a bound for the supremum. Increasing
$\rho$ if necessary, we can assume $\rho\geq \beta$. Summing those
bounds, one obtains
  \begin{align*}
  \Bigl|K_p(x_p,\dotsc) &- \sum_{i=0}^{p-1} \int f_i \dd\mu- K(x_*,\dotsc,x_*, x_p,\dotsc)\Bigr|
  \\
  &\leq C\sum_{i=0}^{p-1} \rho^{p-i} \sum_{j=0}^i \Lip_j(K) \rho^{i-j}
  \leq C\sum_{j=0}^{p-1} \Lip_j(K) \rho^{p-j} (p-j)
  \\ &
  \leq C'\sum_{j=0}^{p-1} \Lip_j(K) (\rho')^{p-j},
  \end{align*}
for any $\rho'\in (\rho,1)$.

Finally, when one computes the sum of the integrals of $f_i$, there
are again cancelations, leaving only $\int K(y,\dotsc, T^{p-1}y,
x_p,\dotsc)\dd \mu(y)$.
\end{proof}

\subsection{Bilateral subshifts of finite type}

\label{subsec_bilateral}

We consider now $X_{\Z}\subset \Sigma^{\Z}$ the state space of a
topologically mixing bilateral subshift of finite type, together with
an invariant Gibbs measure $\mu_{\Z}$. For two points $x=(\dotsc
x_{-1}x_0 x_1\dotsc)$ and $y=(\dotsc y_{-1}y_0 y_1 \dotsc)$ in
$X_{\Z}$, let $s_{\Z}$ be their bilateral separation time, i.e.,
$\inf \{|n| \st x_n\not=y_n\}$, and define a distance
$d_{\Z}(x,y)=\beta^{s_{\Z}(x,y)}$ for some $\beta < 1$. We denote a
function on $X_{\Z}^n$ by $K_{\Z}(x_0\breakdots x_{n-1})$, to
emphasize the dependence both on the past and the future.

\begin{thm}
\label{thm_subshift_bilateral}
The system $(X_{\Z}, T, \mu_{\Z})$ satisfies an exponential
concentration inequality.
\end{thm}

This is stronger than Theorem~\ref{thm_subshift}, which proves this
statement for functions $K_{\Z}(x_0\breakdots x_{n-1})$ depending
only on the future $(x_i)_0^\infty$ of each variable. We will deduce
Theorem~\ref{thm_subshift_bilateral} from this statement by an
approximation argument, by sending everything far away in the future.

\begin{proof}
Let us first assume that $X_{\Z}$ is the full shift. We fix a
function $K_{\Z}(x_0,\dotsc, x_{n-1})$ depending both on the past and
future of the variables. For $N\in \N$, we define $K_N(x_0 \breakdots
x_{n+N-1})=K_{\Z}(x_N,\dotsc,x_{n+N-1})$. Thanks to the invariance of
the measure, it is equivalent to prove concentration inequalities for
$K_{\Z}$ or $K_N$.

Let us now define a function $\Phi_N : X_{\Z}^{n+N} \to X_{\Z}^{n+N}$
depending only of the future of the variables, and let us write
$\tilde K_N=K_N\circ \Phi_N$. Since this function only depends on the
future, Theorem~\ref{thm_subshift} applies to it.

We set $\Phi_N(x_0,\dotsc,x_{n+N-1})=(y_0,\dotsc,y_{n+N-1})$, where
the $y_i$ are defined inductively as follows. First, let us choose an
arbitrary past $(p)_{-\infty}^{-1}$, and let
$y_0=((p)_{-\infty}^{-1}, (x_0)_0^\infty)$: it only depends on the
future of $x_0$. If $y_0,\dotsc, y_{i-1}$ are already defined, we let
$y_i=((y_{i-1})_{-\infty}^0, (x_i)_0^\infty)$. In other words,
  \begin{equation}
  \label{defyinaif}
  y_i=((p)_{-\infty}^{-1}, (x_0)_0, (x_1)_0,\dotsc, (x_{i-1})_0, (x_i)_0^\infty),
  \end{equation}
with an origin laid on $(x_i)_0$. This defines the function $\Phi_N$,
only depending on the future of the points.

Let us study the Lipschitz constants of $\tilde K_N=K_N\circ \Phi_N$.
If we fix $x_j$ for $j\not=i$ and vary $x_i$, then we change $y_j$
for $j\geq i$, at its coordinate with index $-(j-i)$. Therefore,
  \begin{equation*}
  \Lip_i(\tilde K_N) \leq \sum_{j\geq i} \Lip_j(K_N) \beta^{j-i}.
  \end{equation*}
With Cauchy-Schwarz inequality, we get $\sum \Lip_i(\tilde K_N)^2
\leq C \sum \Lip_i(K_N)^2 = C\sum \Lip_i(K_{\Z})^2$, for some
constant $C$. Applying Theorem~\ref{thm_subshift} to $\tilde K_N$ and
changing variables by $x'=T^N x$, we obtain
  \begin{multline*}
  \int e^{\tilde K_N(T^{-N}x',\dotsc, T^{-1}x', x',\dotsc, T^{n-1}x')} \dd \mu_{\Z}(x')
  \\
  \leq e^{\int \tilde K_N(T^{-N}x',\dotsc, T^{-1}x', x',\dotsc, T^{n-1}x')\dd \mu_{\Z}(x')}
  e^{C\sum_{i=0}^{n-1} \Lip_i(K_{\Z})^2}.
  \end{multline*}
By construction, the function $\tilde K_N(T^{-N}x',\dotsc, T^{-1}x',
x',\dotsc, T^{n-1}x')$ converges to $K_{\Z}(x' \breakdots T^{n-1}x')$
when $N$ tends to infinity. Hence, the previous equation gives the
desired exponential concentration.

When $X_{\Z}$ is not the full shift, there is an additional
difficulty: one can not define $\Phi_N$ as above, since a point
defined in~\eqref{defyinaif} might use forbidden transitions. We
should therefore modify the definition of $\Phi_N$ as follows. For
any symbol $a$ of the alphabet, we fix a legal past $p(a)$ of $a$. We
define $\Phi_N(x_0,\dotsc,x_{N+n-1}) = (y_0,\dotsc, y_{N+n-1})$ by
$y_0=(p( (x_0)_0), (x_0)_0^\infty)$ (this point is admissible). Then,
if the transition from $(x_{i-1})_0$ to $(x_i)_0$ is permitted, we
let $y_i=((y_{i-1})_{-\infty}^0, (x_i)_0^\infty)$, and otherwise we
let $y_i=( p( (x_i)_0), (x_i)_0^\infty)$. Therefore, the points $y_i$
only use permitted transitions. The rest of the argument goes through
without modification.
\end{proof}

\section{Uniform Young towers with exponential tails}

\label{sec_exponential}

There are two different definitions of Young towers, given
respectively in \cite{lsyoung_annals} and \cite{lsyoung_recurrence}.
The difference is on the definition of the separation time: in the
first definition, one considers that the dynamics is expanding at
every iteration, while in the second definition one considers that
the dynamics is expanding only when one returns to the basis of the
tower. Therefore, there is less expansion with the second definition
than with the first one, making it more difficult to handle. We will
say that Young towers in the first sense are uniform, while Young
towers in the second sense are non-uniform. In this section, we work
with the (easier) first definition, which turns out to be the most
interesting when dealing with exponential tails. Here is the formal
definition of a uniform Young tower: it is a space $\Delta$
satisfying the following properties.

\begin{enumerate}
\item This space is partitioned into subsets
    $\Delta_{\alpha,\ell}$ (for $\alpha\in \N$ and $\ell\in
    [0,\phi(\alpha)-1]$, where $\phi$ is an integer-valued return
    time function). The dynamics sends bijectively
    $\Delta_{\alpha,\ell}$ on $\Delta_{\alpha,\ell+1}$ if $\ell <
    \phi(\alpha) - 1$, and $\Delta_{\alpha,\phi(\alpha)-1}$ on
    $\Delta_0\coloneqq \bigcup_\alpha \Delta_{\alpha,0}$.
\item The distance is given by $d(x,y)=\beta^{s(x,y)}$ where
    $\beta<1$ and $s(x,y)$ is the separation time for the whole
    dynamics, i.e., the first $n$ such that $T^n x$ and $T^n y$
    are not in the same element of the partition.
\item There is an invariant probability measure $\mu$ such that
    the inverse $g$ of its jacobian satisfies $|g(x)/g(y)-1|\leq
    C d(Tx, Ty)$ for any $x$ and $y$ in the same element of the
    partition.
\item We have $\gcd(\phi(\alpha)\st \alpha\in \N) = 1$ (i.e., the
    tower is aperiodic).
\end{enumerate}

When the return time function $\phi$ has exponential tails, i.e.,
there exists $c_0>0$ with $\int_{\Delta_0} e^{c_0 \phi} \dd \mu <
\infty$, we say that the tower has exponential tails. We will write
$h(x)=\ell$ if $x\in \Delta_{\alpha,\ell}$: this is the height of the
point in the tower. For $x\in \Delta$, we will also denote by $\pi x$
its projection in the basis, i.e., the unique point $y\in \Delta_0$
such that $T^{h(x)}(y)=x$.

\begin{thm}
\label{thm_exponential}
Let $(\Delta,T,\mu)$ be a uniform Young tower with exponential tails.
It satisfies an exponential concentration inequality: there exists
$C>0$ such that, for any $n\in \N$, for any separately Lipschitz
function $K(x_0,\dotsc, x_{n-1})$,
  \begin{equation}
  \label{main_ineq}
  \int e^{K(x,Tx,\dotsc, T^{n-1}x)} \dd \mu (x) \leq e^{\int K(x,\dotsc, T^{n-1}x)\dd \mu(x)}
  e^{C\sum_{i=0}^{n-1} \Lip_i(K)^2}.
  \end{equation}
\end{thm}

Let us first remark that, for any $\epsilon_0>0$, it is sufficient to
prove the theorem for functions $K$ such that $\Lip_i(K) \leq
\epsilon_0$ for all $i$. Assume indeed that this is the case, and let
us prove the general case. Let $K(x_0,\dotsc, x_{n-1})$ be a
separately Lipschitz function. Let us fix an arbitrary point $x_*$ in
$\Delta$. To any $(x_0,\dotsc,x_{n-1})$, we associate
$(y_0,\dotsc,y_{n-1})$ by $y_i=x_i$ if $\Lip_i(K)\leq \epsilon_0$ and
$y_i=x_*$ otherwise. The function $\tilde K(x_0,\dotsc,
x_{n-1})=K(y_0,\dotsc,y_{n-1})$ satisfies $\Lip_i(\tilde K) \leq
\epsilon_0$ for all $i$. Moreover,
  \begin{equation*}
  |K-\tilde K|\leq \sum_i \Lip_i(K) \mathds{1}(\Lip_i(K)>\epsilon_0)
  \leq \sum_i \Lip_i(K)^2/\epsilon_0.
  \end{equation*}
Therefore, the inequality~\eqref{main_ineq} for $\tilde K$ readily
implies the same inequality for $K$, with a different constant
$C'=C+2/\epsilon_0$.

Let us now fix a suitable $\epsilon_0$ (the precise conditions will
be given in the proof of Lemma~ \ref{lem_inductive_exponential}), and
let us consider a function $K$ with $\Lip_i(K)\leq \epsilon_0$ for
all $i$. To prove the exponential concentration inequality, we follow
the strategy of Section~\ref{sec_subshift}. Let
$K_p(x_p,\dotsc)=\E(K|\boF_p)(x_p,\dotsc)$, the first step is to
prove an analogue of Lemma~\ref{lem_controle_Kp_fullshift}. Since the
transfer operator has a spectral gap on a suitable space of
functions, as shown by Young in \cite{lsyoung_annals}, we can easily
mimic the proof of this lemma.
\begin{lem}
\label{lem_controle_Kp}
For all $x_p\in \Delta_0$,
  \begin{equation*}
  \left|K_p(x_p,\dotsc) -\int K(y,\dotsc, T^{p-1}y, x_p,\dotsc)\dd \mu(y)\right|
  \leq C \sum_{j=0}^{p-1} \Lip_j(K) \rho^{p-1-j},
  \end{equation*}
where $C>0$ and $\rho<1$ only depend on $\Delta$.
\end{lem}
The main difference with the subshift case is that this bound is only
valid for $h(x_p)=0$. It is of course false if $h(x_p)$ is large,
since there is no averaging mechanism in this case.
\begin{proof}
As in the proof of Lemma~\ref{lem_controle_Kp_fullshift}, we write
  \begin{equation*}
  K_p(x_p,\dotsc) = \sum_{i=0}^{p-1} \hatT^{p-i} f_i(x_p) + K(x_*,\dotsc,x_*, x_p,\dotsc),
  \end{equation*}
where the function $f_i$ is bounded by $\Lip_i(K)$, and the Lipschitz
norm of $f_i$ on any partition element is at most $C\sum_{j=0}^i
\Lip_j(K) \rho^{i-j}$ for some $\rho<1$.

Let $\boC$ be the space of function on $\Delta$ such that $|f(x)|\leq
C e^{\epsilon h(x)}$ and $|f(x)-f(y)|\leq C d(x,y)e^{\epsilon h(x)}$
for all $x$, $y$ in the same partition element. Young proves in
\cite{lsyoung_annals} that, if $\epsilon$ is small enough, then
$\hatT$ has a spectral gap on $\boC$: there exist $C>0$ and $\rho<1$
such that $\norm{\hatT^k f - \int f \dd\mu}_\boC \leq C \rho^k
\norm{f}_\boC$.

We obtain $\norm{\hatT^{p-i} f_i - \int f_i \dd\mu}_\boC \leq C \rho^{p-i}
\sum_{j=0}^i \Lip_j(K) \rho^{i-j}$. This bound in $\boC$ gives in
particular a bound on the supremum for points at height $0$, and in
particular at the point $x_p$. Summing those bounds over $i$, we get
the desired result exactly as in the proof of
Lemma~\ref{lem_controle_Kp_fullshift}.
\end{proof}

The next step of the proof is the following lemma. It is here that
the Lipschitz constants $\Lip_j(K)$ should all be bounded by
$\epsilon_0$. As before, let $K_p=\E(K|\boF_p)$, and
$D_p=K_p-K_{p+1}$.
\begin{lem}
\label{lem_inductive_exponential}
There exist $\epsilon_0>0$, $C_1>0$ and $\rho<1$ such that any
function $K(x_0,\dotsc, x_{n-1})$ with $\Lip_j(K)\leq \epsilon_0$ for
all $j$ satisfies, for any $p$,
  \begin{equation*}
  \E(e^{D_p}|\boF_{p+1})(x_{p+1},\dotsc) \leq e^{ C_1\sum_{j=0}^p
  \Lip_j(K)^2 \rho^{p-j}}.
  \end{equation*}
\end{lem}
\begin{proof}
If the height of $x_{p+1}$ is positive, then this point has a unique
preimage $y$, and $D_p(y, x_{p+1},\dotsc)=0$. Therefore, $
\E(e^{D_p}|\boF_{p+1})(x_{p+1},\dotsc) = 1$ and the estimate is
trivial.

Assume now that $h(x_{p+1})=0$. Let us denote by $\{z_\alpha\}$ the
preimages of $x_{p+1}$ under $T$ (with $z_\alpha\in
\Delta_{\alpha,\phi(\alpha)-1}$). Let $A(z)=D_p(z, x_{p+1},\dotsc)$,
we have $\E(e^{D_p}|\boF_{p+1})(x_{p+1},\dotsc) =\sum g(z_\alpha)
e^{A(z_\alpha)}$.

Fix a point $z=z_\alpha$, with height $h\geq 0$. If $h\leq p$,
consider the projection $\pi z$ of $z$ in the basis of the tower.
Since $K_p(z,\dotsc)=K_{p-h}(\pi z,\dotsc, z,\dotsc)$,
Lemma~\ref{lem_controle_Kp} shows that $K_p(z,\dotsc)$ is equal to
$\int K(y,\dotsc, T^{p-h}y, \pi z,\dotsc)\dd\mu(y)$ up to $C
\sum_{j=0}^{p-h-1} \Lip_j(K) \rho^{p-h-1-j}$. Up to an additional
error $\sum_{j=p-h}^{p}\Lip_j(K)$, this is equal to $\int K(y,\dotsc,
T^p y, x_{p+1},\dotsc)\dd\mu(y)$. Applying again
Lemma~\ref{lem_controle_Kp} (but to the point $x_{p+1}$), we obtain
  \begin{equation*}
  |A(z)|=|K_p(z,x_{p+1},\dotsc)-K_{p+1}(x_{p+1},\dotsc)|
  \leq C \sum_{j< p-h}  \Lip_j(K) \rho^{p-h-j}
  + \sum_{j=p-h}^p \Lip_j(K).
  \end{equation*}
This estimate is also trivially true if $h>p$ (by convention, one
sets $\Lip_j(K)=0$ for $j<0$). In particular, since $\sup \Lip_j(K)
\leq \epsilon_0$, we always get $|A(z)|\leq C_0(h+1)\epsilon_0$ for
some $C_0>0$ (independent of the value of $\epsilon_0$). Using the
inequality $(x_1+\dotsb+x_k)^2 \leq k\sum x_i^2$, we get
  \begin{equation}
  \label{bound_Az}
  \begin{split}
  |A(z)|^2 &\leq C\left(\sum_{j< p-h}  \Lip_j(K) \rho^{p-h-j}\right)^2
  + C(h+1) \sum_{j=p-h}^p \Lip_j(K)^2
  \\&\leq C \sum_{j< p-h}  \Lip_j(K)^2 \rho^{p-h-j} + C (h+1) \sum_{j=p-h}^p \Lip_j(K)^2,
  \end{split}
  \end{equation}
where we used Cauchy-Schwarz inequality in the last inequality.

The function $A$ satisfies a neat bound on points $z_\alpha$ with
small height, but it is unbounded on points with large height.
Therefore, Hoeffding-Azuma inequality does not apply (contrary to the
subshift of finite type case). While there are certainly exponential
inequalities in the literature that can handle this situation, it is
simpler to reprove everything since we are not interested in good
constants.

We have $|e^{A}-1-A|\leq A^2 e^{|A|}$ for any real number $A$.
Therefore,
  \begin{equation*}
  \left|\sum_\alpha  g(z_\alpha) (e^{A(z_\alpha)}-1- A(z_\alpha))\right|\leq
  \sum g(z_\alpha) A(z_\alpha)^2 e^{|A(z_\alpha)|}.
  \end{equation*}
In the right hand side, $g(z_\alpha) \leq C\mu(\Delta_{\alpha,0})$ by
bounded distortion, and $|A(z_\alpha)|\leq
C_0\epsilon_0(1+\phi(\alpha))$ as we explained above. Together
with~\eqref{bound_Az}, we get
  \begin{multline*}
  \sum g(z_\alpha) A(z_\alpha)^2 e^{|A(z_\alpha)|}
  \\
  \leq C\sum_{h\geq 0} \mu(\phi =h) e^{C_0 \epsilon_0 h} \left(  \sum_{j< p-h}
  \Lip_j(K)^2 \rho^{p-h-j} + (h+1) \sum_{j=p-h}^p \Lip_j(K)^2 \right).
  \end{multline*}
Since the tower has exponential tails, we have $\mu(\phi=h)\leq
\rho_0^h$ for some $\rho_0<1$. If $\epsilon_0$ is small enough, we
get $\mu(\phi =h) e^{C_0 \epsilon_0 h} \leq \rho_1^h$ for some
$\rho_1<1$. Therefore, in the previous bound, the coefficient of
$\Lip_j(K)^2$ is at most
  \begin{equation*}
  \sum_{h< p-j} \rho_1^h \rho^{p-h-j} + \sum_{h \geq p-j} (h+1) \rho_1^h
  \leq (p-j) \rho_2^{p-j} + \rho_2^{p-j},
  \end{equation*}
for some $\rho_2<1$. This is bounded by $C\rho^{p-j}$ for some
$\rho<1$. Hence, we have proved that
  \begin{equation*}
  \left|\sum_\alpha  g(z_\alpha) (e^{A(z_\alpha)}-1- A(z_\alpha))\right|
  \leq C \sum_{j\leq p} \rho^{p-j} \Lip_j(K)^2.
  \end{equation*}
Since $\sum g(z_\alpha)=1$ and $\sum g(z_\alpha) A(z_\alpha)=0$, the
left hand side if equal to $\left| \sum g(z_\alpha)
e^{A(z_\alpha)}-1\right|$. Finally,
  \begin{align*}
  |\E(e^{D_p}|\boF_{p+1})(x_{p+1},\dotsc)|&
  =\left|\sum g(z_\alpha) e^{A(z_\alpha)}\right|
  \leq 1 + C \sum_{j\leq p} \rho^{p-j} \Lip_j(K)^2
  \\&
  \leq e^{C \sum_{j\leq p} \rho^{p-j} \Lip_j(K)^2}.
  \end{align*}
This concludes the proof.
\end{proof}

\begin{proof}[Proof of Theorem~\ref{thm_exponential}]
Consider a function $K$ with $\Lip_j(K)\leq \epsilon_0$ for all $j$.
Using inductively Lemma~\ref{lem_inductive_exponential}, we get for
any $P$
  \begin{equation*}
  \E\left(e^{\sum_{p=0}^{P-1} D_p} |\boF_P\right)
  \leq e^{ C_1\sum_{p=0}^{P-1} \sum_{j=0}^p
  \Lip_j(K)^2 \rho^{p-j}}
  \leq e^{C \sum \Lip_j(K)^2}.
  \end{equation*}
Since $\sum_{p=0}^{P-1} D_p$ converges to $K-\E(K)$ when $P$ tends to
infinity, we obtain $\E(e^{K-\E(K)}) \leq e^{C \sum \Lip_j(K)^2}$.
This proves the exponential concentration inequality in this case.
The general case follows, as we explained after the statement of the
theorem.
\end{proof}

The exponential concentration inequalities for uniform Young towers
with exponential tails easily extends to invertible situations, as
follows. Consider $T_{\Z}:\Delta_{\Z}\to \Delta_{\Z}$ the natural
extension of such a Young tower, with bilateral separation time
$s_{\Z}$, and distance $d_{\Z}(x,y)=\beta^{s_{\Z}(x,y)}$ for some
$\beta<1$.
\begin{thm}\label{thm_exponential_bis}
The transformation $T_{\Z}$ satisfies an exponential concentration
inequality.
\end{thm}
The proof is exactly the same as the proof of
Theorem~\ref{thm_subshift_bilateral}, exploiting the result for the
non-invertible transformation.

\section{Non-uniform Young towers with polynomial tails}

\label{sec_non_uniform}

In this section, we consider Young towers in the sense of
\cite{lsyoung_recurrence}, i.e., non-uniform Young towers. The
combinatorial definition is the same as in
Section~\ref{sec_exponential}, the difference is on the definition of
the separation time (and therefore of the distance) as follows. Let
$\Delta_0$ be the basis of the tower, let $T_0:\Delta_0\to \Delta_0$
be the induced map on $\Delta_0$ (i.e., $T_0(x)=T^{\phi(x)}(x)$ where
$\phi(x)$ is the return time of $x$ to $\Delta_0$). For $x,y\in
\Delta_0$, let $s(x,y)$ be the smallest integer $n$ such that $T_0^n
(x)$ and $T_0^n (y)$ are not in the same partition element. This
separation time is extended to $\Delta$ as follows. For $x,y\in
\Delta$, let $s(x,y)=s(\pi x,\pi y)$ if $x$ and $y$ are in the same
partition element, and $s(x,y)=0$ otherwise. In other words, $s(x,y)$
is the number of returns to the basis before the trajectories of $x$
and $y$ separate. Finally, the new distance is
$d(x,y)=\beta^{s(x,y)}$ for some $\beta<1$.

Intuitively, we are now considering maps that are expanding only when
one returns to the basis, and can be isometries between successive
returns, while the maps of Section~\ref{sec_exponential} are always
expanding. The setting is not uniformly expanding any more, rather
non-uniformly expanding. For instance, intermittent maps can be
modeled using non-uniform Young towers.

If the tails are not exponential any more, one can not hope to get
exponential concentration inequalities. If the tails have a moment of
order $q\geq 2$, then the moments of order $2q-2$ of Birkhoff sums
are controlled, and this is optimal \cite[Theorem
3.1]{melbourne_nicol_large_deviations}. Our goal in this section is
to generalize this result to a concentration inequality (with the
same optimal moment).

\begin{thm}
\label{thm_nonuniform}
Let $T:\Delta\to \Delta$ be a non-uniform Young tower. Assume that,
for some $q\geq 2$, $\sum \phi(\alpha)^q \mu(\Delta_{\alpha,0}) <
\infty$. Then $T$ satisfies a polynomial concentration inequality
with moment $2q-2$, i.e., there exists a constant $C>0$ such that,
for any $n\in\N$, for any separately Lipschitz function
$K(x_0,\dotsc, x_{n-1})$,
  \begin{equation*}
  \int \left|K(x,\dotsc, T^{n-1}x) -\int K(y,\dotsc, T^{n-1}y)\dd\mu(y)\right|^{2q-2}\dd\mu(x)
  \leq C\left(\sum_j \Lip_j(K)^2 \right)^{q-1}.
  \end{equation*}
\end{thm}

The proof is considerably more difficult than the arguments in the
previous section (and also than the arguments of
\cite{melbourne_nicol_large_deviations} since the main inequality
these arguments rely on, due to Rio, is of no help in our situation).
The general strategy is the same as in the previous sections:
decompose $K-\E(K)$ as $\sum D_p$ where $D_p$ is a martingale
difference sequence, obtain good estimates on $D_p$, and then apply a
martingale inequality (in our case, the Rosenthal-Burkholder
inequality) to obtain a bound on $K-\E(K)$. The difficulty comes from
the non-uniform expansion of the map: instead of a uniformly decaying
geometric series as in the previous sections, our estimates will be
non-uniform, quantified by the number of visits to the basis in a
definite amount of time.

The rest of this section is devoted to the proof of
Theorem~\ref{thm_nonuniform}. In particular, we will always assume
that $\Delta$ is a non-uniform Young tower satisfying $\sum
\phi(\alpha)^q \mu(\Delta_{\alpha,0}) < \infty$ for some $q\geq 2$.

\begin{rmk}
The arguments below also give an exponential concentration inequality
in non-uniform Young towers with exponential tails, thereby
strengthening Theorem ~\ref{thm_exponential}. Since most interesting
Young towers with exponential tails are uniform, we will not give
further details in this direction.
\end{rmk}

\subsection{Notations}

As usual, the letter $C$ denotes a constant that may change from one
occurrence to the next. Let us also introduce a similar notation for
sequences. For $Q\geq 0$, we will write $c_n^{(Q)}$ for a sequence of
nonnegative numbers such that $\sum n^Q c_n^{(Q)} < \infty$, and we
will allow this sequence to change from one line to the other (or
even on the same line). We will also write $d_n^{(Q)}$ for a generic
nondecreasing sequence with $\sum n^Q d_n^{(Q)} < \infty$.

If $u_n$ and $v_n$ are sequences, their convolution $u\star v$ is
given by $(u\star v)_n=\sum_{k=0}^n u_k v_{n-k}$. One easily checks
that, for $Q, Q'\geq 0$,
  \begin{equation}
  \label{convol_OK}
  (c^{(Q)} \star c^{(Q')})_n \leq c^{(\min(Q,Q'))}_n.
  \end{equation}
Following the above convention, this statement should be understood
as follows: if two sequences $u$ and $v$ satisfy, respectively, $\sum
n^Q u_n < \infty$ and $\sum n^{Q'}v_n < \infty$, then $w=u\star v$
satisfies $\sum n^{\min(Q,Q')} w_n < \infty$. Indeed, letting
$Q''=\min(Q,Q')$,
  \begin{align*}
  \sum n^{Q''} w_n & =\sum_{k, \ell} (k+\ell)^{Q''} u_k v_\ell
  \leq \sum_{k,\ell} (k+1)^{Q''} (\ell+1)^{Q''} u_k v_\ell
  \\&
  \leq \left(\sum (k+1)^Q u_k\right)\cdot \left( \sum (\ell+1)^{Q'}v_\ell\right)
  <\infty.
  \end{align*}

We also have for $Q\geq 1$
  \begin{equation}
  \label{serie_maj}
  \sum_{k=n}^\infty c_k^{(Q)} \leq d^{(Q-1)}_n.
  \end{equation}
Indeed,
  \begin{equation*}
  \sum n^{Q-1} \sum_{k=n}^\infty c_k^{(Q)}
  =\sum_k \left(\sum_{n=0}^k n^{Q-1}\right) c_k^{(Q)}
  \leq \sum_k C k^Q c_k^{(Q)}
  < \infty,
  \end{equation*}
and the sequence $\sum_{k=n}^\infty c_k^{(Q)}$ is nonincreasing.

\subsection{Renewal sequences of operators, estimates on the returns
to the basis}
\label{subsec_renewal}

An important tool for our study will be renewal sequences of
operators, as developed by Sarig and Gou\"ezel \cite{sarig_decay,
gouezel_decay, gouezel_these}, that we will now quickly describe.

Consider a function $f$, we wish to understand $\hatT^n
f(x)=\sum_{T^n y=x} g^{(n)}(y)f(y)$ for $x\in\Delta_0$. For a
preimage $y$ of $x$ under $T^n$, we can consider its first entrance
into $\Delta_0$, and then its successive returns to $\Delta_0$. We
obtain a decomposition
  \begin{equation}
  \label{def_Tn}
  1_{\Delta_0} \hatT^n = \sum_{k+b=n}T_k B_b,
  \end{equation}
where $T_k$ takes the successive returns to $\Delta_0$ (during time
$k$) into account, and $B_b$ deals with the part of the trajectory
outside $\Delta_0$. Formally, for $x\in \Delta_0$, $T_k f(x)=\sum
g^{(k)}(y)f(y)$ where the sum is restricted to those $y$ such that
$T^k y=x$ and $y\in \Delta_0$. The operator $B_b$, in turn, is given
on $\Delta_0$ by $B_b f(x)=\sum g^{(b)}(y)f(y)$ where the sum is
restricted to those $y$ with $T^b y=x$ and $y,\dotsc, T^{b-1}y\not\in
\Delta_0$.

The operators $B_b$ are essentially trivial to understand, their
behavior being controlled by the tails of the return time function
$\phi$. On the other hand, the operators $T_k$ embody most of the
dynamics of the transformation. To understand them, we introduce yet
another operator $R_j$ considering only the first return to
$\Delta_0$ at time $j$, i.e., $R_j f(x)=\sum g^{(j)}(y)f(y)$ where
the sum is restricted to those $y$ such that $T^j y=x$ and $y\in
\Delta_0$, $Ty,\dotsc, T^{j-1}y\not\in \Delta_0$. Splitting a
trajectory into its successive excursions outside of $\Delta_0$, one
obtains
  \begin{equation*}
  T_k=\sum_{\ell\geq 1} \sum_{j_1+\dotsb+j_\ell=k} R_{j_1}\dotsm R_{j_\ell}.
  \end{equation*}
Formally, this equation can be written as
  \begin{equation}
  \label{renewal_equation}
  \sum T_k z^k=(I-\sum R_j z^j)^{-1}.
  \end{equation}
In fact, the series defined in this equation are holomorphic for
$|z|<1$ (as operators acting on the space $\boC$ of Lipschitz
functions on $\Delta_0$) and this equality is a true equality between
holomorphic functions. Moreover, the spectral radius of $\sum R_j
z^j$ is at most $1$ for $|z|\leq 1$.

A powerful way to use the previous equality is Banach algebra
techniques. Simple examples of Banach algebras are given by Banach
spaces $\boB$ of sequences $c_n$ such that, if $(c_n)_{n\in\N} \in
\boB$ and $(c'_n)_{n\in\N} \in \boB$, then their convolution $c\star
c'$ still belongs to $\boB$. For instance, this is the case of
sequences with a moment of order $Q\geq 0$ (by~\eqref{convol_OK}), or
of sequences satisfying $c_n=O(1/n^Q)$ for some $Q>1$. Given such a
Banach algebra of sequences $\boB$, one can consider the Banach
algebra $\tilde\boB$ of sequences of operators $(M_n)_{n\in\N} $
(acting on some fixed Banach space $\boC$) such that the sequence
$(\norm{M_n})_{n\in\N} $ belongs to $\boB$. One easily checks that
$\tilde\boB$ is again a Banach algebra (for the convolution product).

When the Banach algebra of sequences $\boB$ satisfies a technical
condition (its characters should all be given by evaluation of the
power series $\sum c_n z^n$ at a point $z$ of the unit disk), which
is satisfied in all examples we mentioned above, then one can use the
Wiener lemma to obtain the following property: if a sequence of
operators $(M_n)_{n\in\N} $ belongs to $\tilde\boB$ and $\sum M_n
z^n$ is invertible as an operator on $\boC$ for any $z$ in the closed
unit disk, then $(M_n)_{n\in\N} $ is invertible in $\tilde\boB$. In
particular, the power series $\sum M'_n z^n = (\sum M_n z^n)^{-1}$
satisfies $(\norm{M'_n})_{n\in\N}  \in \boB$.

Using Banach algebra arguments and the renewal
equation~\eqref{renewal_equation}, the following proposition is
proved in \cite[Proposition 2.2.19]{gouezel_these}.
\begin{prop}
\label{prop_wiener}
Consider a Banach algebra of sequences $\boB$ satisfying several
technical conditions. If the sequence $(\sum_{k>n}
\mu(\phi=k))_{n\in\N} $ belongs to $\boB$, then this is also the case
of the sequence $(\norm{T_{n+1}-T_n})_{n\in\N} $. Moreover, $T_n$
converges to $\Pi : f\mapsto (\int_{\Delta_0} f) 1_{\Delta_0}$.
\end{prop}
The technical conditions on the Banach algebra (all the characters of
$\boB$ should be given by the evaluation at a point of the closed
unit disk, and the symmetrized version of $\boB$ should contain the
Fourier coefficients of partitions of unity of the circle) will not
be important for us, let us only mention that they are satisfied for
the Banach algebras of series with moments of order $Q\geq 0$.

\medskip

The contraction properties of the dynamics $T$ are dictated by the
number of returns to the basis. Their asymptotics are estimated in
the next lemma.

\begin{lem}
\label{lem_Psin}
For $x\in \Delta$, let $\psi_n(x)=\Card\{0\leq k\leq n-1 \st T^k x\in
\Delta_0\}$ be the number of visits to the basis of $x$ before time
$n$, and let $\Psi_n(x)=\rho^{\psi_n(x)}$, where $\rho<1$. If the
return time on $\Delta_0$ has a moment of order $q\geq 1$ (i.e.,
$\mu(\phi =n)\leq c_n^{(q)}$), we have
  \begin{equation*}
  \int_{T^{-n}\Delta_0} \Psi_n \dd \mu(x) \leq c_n^{(q-1)}.
  \end{equation*}
\end{lem}
This bound is optimal: on $\Delta_{\alpha, \phi(\alpha)-n}$ (for
$\alpha$ with $\phi(\alpha)>n$), we have $\Psi_n=1$. Therefore, the
integral in the lemma is bounded from below by
$\mu(\bigcup_{\phi(\alpha)>n} \Delta_{\alpha, 0}) \sim
\sum_{n+1}^\infty c_k^{(q)} \sim c_n^{(q-1)}$.
\begin{proof}
Let us define an operator $U_n$ by the series $\sum U_n
z^n=\sum_{k=0}^\infty (\rho \sum R_n z^n)^k=(I-\rho \sum R_n
z^n)^{-1}$. Then $U_n f(x)=\sum g^{(n)}(y) \Psi_n(y) f(y)$, where the
sum is restricted to those $y\in \Delta_0$ with $T^n y = x$.
Integrating and changing variables, we obtain
  \begin{equation*}
  \int_{\Delta_0} U_n 1(x)\dd \mu(x)
  =\int_{\Delta_0 \cap T^{-n}(\Delta_0)}\Psi_n(y) \dd \mu(y).
  \end{equation*}

Since the spectral radius of $\sum R_n z^n$ is at most $1$ for
$|z|\leq 1$, it follows that $I-\rho \sum R_n z^n$ is invertible on
$\boC$ (since $\rho<1$). Moreover, the sequence $\norm{R_n}$
satisfies $\norm{R_n} \leq C \mu(\phi=n)\leq c_n^{(q)}$. It follows
from Wiener's Lemma that $\sum U_n z^n = (I-\rho \sum R_n z^n)^{-1}$
belongs to the same Banach algebra of operators, i.e.,
$\norm{U_n}\leq c_n^{(q)}$. We obtain
 \begin{equation*}
  \int_{\Delta_0\cap T^{-n}\Delta_0} \Psi_n(y) \dd \mu (y) \leq c_n^{(q)}.
  \end{equation*}
To study the integral of $\Psi_n$ on $T^{-n}\Delta_0$, denote by
$\Lambda_b$ the set of points in $\Delta$ that enter $\Delta_0$
exactly at time $b$. On $\Lambda_b$, we have
$\Psi_n(y)=\Psi_{n-b}(T^b y)$. A distortion control gives
  \begin{equation*}
  \int_{\Lambda_b \cap T^{-n}\Delta_0} \Psi_n
  \leq C \mu(\Lambda_b) \int_{\Delta_0 \cap T^{-(n-b)}\Delta_0} \Psi_{n-b}
  \leq C \mu(\Lambda_b)c_{n-b}^{(q)}.
  \end{equation*}
Moreover, for $b>0$, $\Lambda_b=\bigcup_{\phi(\alpha)\geq b}
\Delta_{\alpha, \phi(\alpha)-b}$, hence $\mu(\Lambda_b)\leq
\sum_{\ell\geq b} c_\ell^{(q)} \leq c_b^{(q-1)}$. We obtain
  \begin{equation*}
  \int_{T^{-n}\Delta_0}  \Psi_n(y) \dd \mu(y)
  = \sum_{b=0}^n \int_{\Lambda_b \cap T^{-n}\Delta_0} \Psi_n(y)\dd\mu(y)
  \leq C\sum_{b=0}^n c_b^{(q-1)} c_{n-b}^{(q)}.
  \end{equation*}
By~\eqref{convol_OK}, this is bounded by $c_n^{(q-1)}$.
\end{proof}

\subsection{Bounding \texorpdfstring{$D_p$}{Dp}}

To follow the same strategy as in the previous sections, we need to
show that $K_p$ is close to an integral, as in
Lemma~\ref{lem_controle_Kp_fullshift}. To do so, as in the proof of
this lemma, we define a function $f_i$ as in~\eqref{eq_define_fi},
and control its iterates under the transfer operator. The first step
is to control its Lipschitz constant.

\begin{lem}
\label{lem_regularite_fi}
For $z$ and $z'$ with zero height, $|f_i(z)|\leq C \Lip_i(K)$ and
  \begin{equation*}
  |f_i(z)-f_i(z')|\leq C d(z,z') \sum_{j=0}^i \Lip_j(K)c_{i-j}^{(q-1)}.
  \end{equation*}
\end{lem}
\begin{proof}
The inequality $|f_i(z)|\leq C \Lip_i(K)$ is trivial. To control the
Lipschitz constant, as in~\eqref{wpoiuxcvjl,mlkxwc}, we decompose
  \begin{equation*}
  \begin{split}
  f_i(z)-f_i(z')={}&\sum (g^{(i)}(y)-g^{(i)}(y')) H(y,\dotsc, T^i y)
  \\& +
  \sum g^{(i)}(y')(H(y,\dotsc, T^i y)-H(y',\dotsc, T^i y')).
  \end{split}
  \end{equation*}
Using distortion controls, we bound the first sum by $C \Lip_i(K)
d(z,z')$. For the second sum, we replace successively each $T^j y$
with $T^j y'$, writing it as
  \begin{multline*}
  \sum_{T^i y' = z'} \sum_{j=0}^i g^{(i)}(y') (H(y,\dotsc, T^{j-1}y, T^j y, T^{j+1}y',\dotsc, T^i y')
  \\
  - H(y,\dotsc, T^{j-1}y, T^j y', T^{j+1}y',\dotsc, T^i y')).
  \end{multline*}
Since the distance between $T^j y$ and $T^j y'$ is bounded by
$\Psi_{i-j}(T^j y') d(z,z')$, we obtain a bound
  \begin{align*}
  \sum_{T^i y'=z'} & \sum_{j=0}^i g^{(i)}(y') \Psi_{i-j}(T^j y') \Lip_j(K) d(z,z')
  \\&\leq d(z,z')\sum_{j=0}^i \sum_{T^{i-j}(y'_j) = z'} g^{(i-j)}(y'_j) \Psi_{i-j}(y'_j) \Lip_j(K)
  \\&\leq C d(z,z') \sum_{j=0}^i \Lip_j(K) \int_{T^{-(i-j)}\Delta_0} \Psi_{i-j},
  \end{align*}
by bounded distortion. With Lemma~\ref{lem_Psin}, this gives the
result.
\end{proof}

To follow the strategy of proof of
Lemma~\ref{lem_controle_Kp_fullshift}, we need to understand the
iterates of $f_i$ under the transfer operator. This is done in the
next lemma.

\begin{lem}
\label{lem_Lrfi}
For any $r\geq 0$ and any $z\in \Delta_0$, we have
  \begin{equation*}
  \left|\hatT^r f_i(z)-\int_{\Delta} f_i\right|\leq \sum_{j=0}^i \Lip_j(K)
  \left(\sum_{k=0}^r c_k^{(q-2)}c_{i-j+r-k}^{(q-1)}\right).
  \end{equation*}
\end{lem}
\begin{proof}
We will use the decomposition $1_{\Delta_0}\hatT^r = \sum_{k+b=r} T_k
B_b$ given by~\eqref{def_Tn} to understand $\hatT^r f_i$.

Let us first describe the asymptotics of $T_k$. Let $\boC$ denote the
space of Lipschitz functions on the basis $\Delta_0$ of the tower. We
define an operator $\Pi$ on $\boC$ by $\Pi f=(\int_{\Delta_0} f)
1_{\Delta_0}$. The operators $T_n$ converge to $\Pi$. Since
$\norm{T_n - T_{n+1}}\leq c_n^{(q-1)}$ by
Proposition~\ref{prop_wiener}, we have
  \begin{equation}
  \label{TkOK_optimal}
  \norm{T_k-\Pi}  \leq \sum_{n=k}^\infty \norm{T_n-T_{n+1}}
  \leq  \sum_{n=k}^\infty  c_n^{(q-1)}
  \leq c_k^{(q-2)},
  \end{equation}
by~\eqref{serie_maj}.

\smallskip

We will now estimate $\norm{B_b f_i}_{\boC}$ using
Lemma~\ref{lem_regularite_fi}. For $z\in \Delta_0$, we have
  \begin{equation*}
  B_b f_i(z)=\sum_{\phi(\alpha)\geq b} g^{(b)}(z_\alpha) f_i(z_\alpha),
  \end{equation*}
where $z_\alpha$ is the unique preimage of $z$ under $T^b$ in
$\Delta_{\alpha, \phi(\alpha)-b}$. We have
  \begin{equation}
  \label{borneBb_sup}
  |B_b f_i|_\infty\leq
  |f_i|_\infty \cdot C \sum_{\phi(\alpha)\geq b}\mu(\Delta_{\alpha,0})
  \leq C|f_i|_\infty c_b^{(q-1)}
  \leq C\Lip_i(K) c_b^{(q-1)}.
  \end{equation}
Let us now estimate $B_b f_i(z)-B_b f_i(z')$ for $z$ and $z'$ in the
same partition element. If we form the difference
$g^{(b)}(z_\alpha)-g^{(b)}(z'_\alpha)$, the resulting term is bounded
by $Cd(z,z') \Lip_i(K) c_b^{(q-1)}$ (using distortion controls and
the same computation as in~\eqref{borneBb_sup}). On the other hand,
denoting by $h_\alpha=\phi(\alpha)-b$ the height of $z_\alpha$, we
have
  \begin{equation*}
  |f_i(z_\alpha)- f_i(z'_\alpha)| \leq C\left(\sum_{j=0}^{i-h_\alpha} \Lip_j(K) c_{i-j-h_\alpha}^{(q-1)}
  + \sum_{j=i-h_\alpha+1}^i \Lip_j(K)\right) d(z,z').
  \end{equation*}
This follows from Lemma~\ref{lem_regularite_fi} applied to the
function $f_{i-h_\alpha}$ and the points $\pi z_\alpha$ and $\pi
z'_\alpha$. Summing over $\alpha$, we obtain a bound for the
Lipschitz constant of $B_b f_i$ of the form
  \begin{equation*}
  \sum_{\phi(\alpha) \geq b}  g^{(b)}(z_\alpha) \left[\sum_{j=0}^{i- h_\alpha}
	\Lip_j(K)c^{(q-1)}_{i-j-h_\alpha}
  + \sum_{j=i-h_\alpha+1}^i \Lip_j(K) \right].
  \end{equation*}
By bounded distortion, $g^{(b)}(z_\alpha) \leq C
\mu(\Delta_{\alpha,0})$. Taking the union over $\alpha$ and writing
$\ell=\phi(\alpha)$, we get that the coefficient of $\Lip_j(K)$ in
this sum is bounded by
  \begin{equation*}
  C \sum_{\ell = b}^{b + i-j} \mu(\phi=\ell) c^{(q-1)}_{i-j- (\ell-b)}
  + C \sum_{\ell=b+i-j+1}^\infty \mu(\phi=\ell).
  \end{equation*}
The second term is bounded by $c_{i-j+b}^{(q-1)}$
by~\eqref{serie_maj}, while the first term is bounded by
  \begin{equation*}
  \sum_{\ell=0}^{i-j+b } c_\ell^{(q)} c_{i-j+b - \ell}^{(q-1)}
  \leq c^{(q-1)}_{i-j+b}
  \end{equation*}
by~\eqref{convol_OK}. We have shown that
  \begin{equation*}
  \norm{B_b f_i}_{\boC} \leq \sum_{j=0}^i \Lip_j(K) c_{i-j+b}^{(q-1)}.
  \end{equation*}
(The contribution of~\eqref{borneBb_sup} is compatible with this
bound.)

\smallskip

Let us now study $\hatT^r f_i$ on $\Delta_0$. We write $T_k=\Pi+E_k$
with $\norm{E_k}\leq c_k^{(q-2)}$, by~\eqref{TkOK_optimal}. Hence,
  \begin{equation}
  \label{eqLr_opt}
  \hatT^r f_i = \sum_{k+b=r}T_k B_b f_i = \sum_{k+b=r} \Pi B_b f_i
  + \sum_{k+b=r} E_k B_b f_i.
  \end{equation}
The first term is a constant function equal to $\sum_{b=0}^r
\int_{\Delta_0} B_b f_i$. Denoting by $\Lambda_b$ the set of points
that enter $\Delta_0$ exactly at time $b$, we have $\int_{\Delta_0}
B_b f_i =\int_{\Lambda_b} f_i$. As a consequence
  \begin{equation*}
  \sum_{b=0}^r \int_{\Delta_0} B_b f_i -\int f_i = -\int_{\bigcup_{b>r}\Lambda_b} f_i
  \leq |f_i|_\infty \sum_{b>r} \mu(\Lambda_b)
  \leq \Lip_i(K) \sum_{b>r} c_b^{(q-1)}
  \leq \Lip_i(K) c_r^{(q-2)},
  \end{equation*}
by~\eqref{serie_maj}. This bound is compatible with the statement of
the lemma. The second term of~\eqref{eqLr_opt} is bounded (in $\boC$
norm, thus in sup norm) by
  \begin{equation*}
  \sum_{k+b=r} c_k^{(q-2)} \norm{B_b f_i}_{\boC}
  \leq \sum_{j=0}^i \Lip_j(K)\cdot \sum_{k+b=r} c_k^{(q-2)}c_{i-j+b}^{(q-1)}.
  \end{equation*}
This proves the lemma.
\end{proof}

We can now obtain the following lemma, which is the analogue in our
setting of Lemma ~\ref{lem_controle_Kp_fullshift}.
\begin{lem}
\label{control_Kp_non_uniform}
For all $x_p\in \Delta_0$,
  \begin{equation*}
  \left|K_p(x_p,\dotsc) -\int K(y,\dotsc, T^{p-1}y, x_p,\dotsc)\dd \mu(y)\right|
  \leq \sum_{j=0}^{p-1} \Lip_j(K) c_{p-j}^{(q-2)}.
  \end{equation*}
\end{lem}
\begin{proof}
Just like in the proof of Lemma~\ref{lem_controle_Kp_fullshift}
  \begin{equation*}
  \left|K_p(x_p,\dotsc)-\int K(y,\dotsc, T^{p-1}y, x_p,\dotsc)\right|
  \leq \sum_{i=0}^{p-1} \left|\hatT^{p-i}f_i(x_p)-\int f_i\right|.
  \end{equation*}
By Lemma ~ \ref{lem_Lrfi}, this quantity is bounded by
  \begin{equation*}
  C\sum_{i=0}^{p-1} \sum_{j=0}^i \Lip_j(K) \left( \sum_{k=0}^{p-i} c_k^{(q-2)}c_{i-j+p-i-k}^{(q-1)}\right).
  \end{equation*}
The coefficient of $\Lip_j(K)$ in this sum is
  \begin{equation*}
  \sum_{k=0}^{p-j} c_k^{(q-2)} (p-k-j) c_{p-k-j}^{(q-1)}
  \leq \sum_{k=0}^{p-j} c_k^{(q-2)}c_{p-k-j}^{(q-2)}
  \leq c_{p-j}^{(q-2)}
  \end{equation*}
by~\eqref{convol_OK}. This proves the lemma.
\end{proof}

The previous lemma makes it possible to control the moments of
$D_p=K_p-K_{p+1}$:
\begin{lem}
\label{lem_devroye_Kp_opt_Lq}
For all $\kappa\leq 2q$,
  \begin{equation*}
  \E( |D_p|^\kappa |\boF_{p+1})(x_{p+1},\dotsc)
  \leq C\sum_{j=0}^p
  \Lip_j(K)^\kappa c_{p-j}^{(q-2)} + C \sum_{h\geq 0}c_h^{(q-\kappa/2)}
  \left(\sum_{j=p-h+1}^{p}\Lip_j(K)^2\right)^{\kappa/2}.
  \end{equation*}
\end{lem}
\begin{proof}
We follow closely the strategy of the proof of
Lemma~\ref{lem_inductive_exponential}. If the height of $x_{p+1}$ is
positive, the estimate is trivial. Otherwise, let $\{z_\alpha\}$
denote the preimages of $x_{p+1}$ under $T$, with respective height
$h_\alpha=\phi(\alpha)-1$. Let $A(z)=D_p(z, x_{p+1},\dotsc)$, we have
$\E(|D_p|^\kappa|\boF_{p+1})(x_{p+1},\dotsc)=\sum g(z_\alpha)
|A(z_\alpha)|^\kappa$.

Fix a point $z=z_\alpha$ with height $h\geq 0$. If $h\leq p$,
consider the projection $\pi z$ of $z$ in the basis of the tower.
Using Lemma~\ref{control_Kp_non_uniform} (at time $p-h$ for the point
$\pi z$, and at time $p+1$ for the point $x_{p+1}$), we get
  \begin{equation}
  \label{eq_mainineq_A}
  |A(z)|\leq \sum_{j\leq p-h} \Lip_j(K) c_{p-h-j}^{(q-2)} + \sum_{j=p-h+1}^p \Lip_j(K).
  \end{equation}
This estimate also holds (trivially) if $h>p$.

To estimate $|A(z)|^\kappa$, we first use the inequality
$(x+y)^\kappa \leq C x^\kappa+Cy^\kappa$ to separate the two sums.
Then, in the first sum, since $c_{p-h-j}^{(q-2)}$ is summable, we may
use H\"older inequality to get $\left(\sum_{j\leq p-h} \Lip_j(K)
c_{p-h-j}^{(q-2)}\right)^\kappa \leq C \sum_{j\leq p-h}
\Lip_j(K)^\kappa c_{p-h-j}^{(q-2)}$. For the second sum, we write
$\left(\sum_{j=p-h+1}^{p}\Lip_j(K)\right)^2 \leq h
\sum_{j=p-h+1}^{p}\Lip_j(K)^2$, and we obtain
  \begin{equation*}
  |A(z)|^\kappa \leq \sum_{j\leq p-h} \Lip_j(K)^\kappa  c_{p-h-j}^{(q-2)}
  + C h^{\kappa/2} \left(\sum_{j=p-h+1}^p \Lip_j(K)^2\right)^{\kappa/2}.
  \end{equation*}
Summing over $\alpha$, we get that $\sum g(z_\alpha)
|A(z_\alpha)|^\kappa$ is at most
  \begin{equation*}
  C\sum_{h=0}^\infty \mu(\phi=h) \left(\sum_{j\leq p-h} \Lip_j(K)^\kappa  c_{p-h-j}^{(q-2)}
  + h^{\kappa/2} \left(\sum_{j=p-h+1}^p \Lip_j(K)^2\right)^{\kappa/2}\right).
  \end{equation*}
In the first sum, the coefficient of $\Lip_j(K)^\kappa$ is at most
  \begin{equation*}
  \sum_{h=0}^{p-j} c_h^{(q)} c_{p-h-j}^{(q-2)}
  \leq c_{p-j}^{(q-2)}
  \end{equation*}
by~\eqref{convol_OK}. In the second sum, $\mu(\phi=h)h^{\kappa/2}
\leq c_h^{(q-\kappa/2)}$, yielding the statement of the lemma.
\end{proof}

\subsection{Proof of Theorem~\ref{thm_nonuniform}}

\label{subsec_proof_nonuniform}

We will use the following Rosenthal-Burkholder martingale inequality
\cite[Theorem 21.1 and Inequality (21.5)]{burkholder}. Let $\boF_p$
be a decreasing sequence of $\sigma$-algebras, and let $D_p$ be a
sequence of reverse martingale difference with respect to $\boF_p$
(i.e., $D_p$ is $\boF_p$-measurable and $\E(D_p|\boF_{p+1})=0$). For
all $Q\geq 2$,
  \begin{equation*}
  \norm{\sum D_p}^Q_{L^Q} \leq C \E\left( \left[\sum_p \E(D_p^2 | \boF_{p+1})\right]^{Q/2}\right)
  +C \sum_p \E( |D_p|^Q).
  \end{equation*}

We apply this inequality to $\boF_p$ the $\sigma$-algebra of sets
depending only on $x_p,x_{p+1},\dotsc$, to $D_p = K_p-K_{p+1}$ and to
$Q=2q-2$. By Lemma~\ref{lem_devroye_Kp_opt_Lq} with $\kappa=2$, we
have
  \begin{equation}
  \label{pouwxcvlmkj}
  \E(D_p^2|\boF_{p+1})(x_{p+1},\dotsc) \leq C\sum_{j=0}^p \Lip_j(K)^2 c_{p-j}^{(q-2)}
  +C\sum_{h\geq 0} c_h^{(q-1)} \sum_{j=p-h+1}^{p}\Lip_j(K)^2.
  \end{equation}
The coefficient of $\Lip_j(K)^2$ in this estimate is bounded by
$c_{p-j}^{(q-2)} + \sum_{h \geq p-j+1} c_h^{(q-1)} \leq
c_{p-j}^{(q-2)}$. Hence, the first term in Rosenthal-Burkholder
inequality is bounded by
  \begin{equation*}
  C\left(\sum_{p} \sum_{j=0}^p \Lip_j(K)^2 c_{p-j}^{(q-2)}\right)^{q-1}
  \leq C \left(\sum_j \Lip_j(K)^2 \right)^{q-1}.
  \end{equation*}
For the second term, we should bound $\sum_p \E(|D_p|^{2q-2})$. We
sum the estimates of Lemma ~ \ref{lem_devroye_Kp_opt_Lq} (with
$\kappa=2q-2$), to get
  \begin{equation}
  \label{wopxiucvmlwxjcv}
  \sum_p \E(|D_p|^{2q-2})
  \leq C \sum_{j}\sum_{p\geq j} \Lip_j(K)^{2q-2} c_{p-j}^{(q-2)}
  + C \sum_{h\geq 0} c_h^{(1)} \sum_p \left(\sum_{j=p-h+1}^{p}\Lip_j(K)^2\right)^{q-1}.
  \end{equation}
In the first sum, the coefficient of $\Lip_j(K)^{2q-2}$ is $\sum_k
c_k^{(q-2)}\leq C$, therefore this sum is bounded by $C\sum_j
\Lip_j(K)^{2q-2} \leq C\left( \sum \Lip_j(K)^2 \right)^{q-1}$.

The second sum is more delicate. Let us fix $h$ and $p_0 \in [0, h)$,
and let us consider the contribution of those $p$ in $p_0+\Z h$. The
intervals $[p-h+1, p]$ are disjoint. The inequality $\sum x_i^{q-1}
\leq (\sum x_i)^{q-1}$ yields
  \begin{equation*}
  \sum_{p\equiv p_0 \; [h]} \left(\sum_{j=p-h+1}^{p}\Lip_j(K)^2\right)^{q-1}
  \leq \left( \sum_{p\equiv p_0 \; [h]} \sum_{j=p-h+1}^{p}\Lip_j(K)^2\right)^{q-1}
  \leq \left(\sum_j \Lip_j(K)^2\right)^{q-1}.
  \end{equation*}
Summing over the $h$ possible value of $p_0$, we get that the second
sum of~\eqref{wopxiucvmlwxjcv} is bounded by
  \begin{equation*}
  C\sum_{h\geq 0} c_h^{(1)} h \left(\sum_j \Lip_j(K)^2\right)^{q-1}
  \leq C \left(\sum_j \Lip_j(K)^2\right)^{q-1},
  \end{equation*}
since $\sum h c_h^{(1)} <\infty$ by definition.

We have proved that $\norm{\sum D_p}_{L^{2q-2}}^{2q-2} \leq C\left(
\sum_j \Lip_j(K)^2\right)^{q-1}$. Since $\sum D_p=K-\E(K)$, this
proves Theorem~\ref{thm_nonuniform}. \qed

\section{Invertible non-uniform Young towers}
\label{sec_non_uniform_invertible}

Let $T:X\to X$ be a non-uniform Young tower, with invariant measure
$\mu$. Its natural extension $T_{\Z}:X_{\Z}\to X_{\Z}$ preserves a
probability measure $\mu_{\Z}$. There is a natural distance on
$X_{\Z}$, defined as follows. First, the positive separation time
$s(x,y)$ is defined as for $T$. One can also define a negative
separation time $s_-(x,y)$ in the same way, but towards the past: one
iterates towards the past until the points are in different elements
of the Markov partition, and one counts the number of visits to
$\Delta_0$ in between. The distance $d_{\Z}$ is then defined by
$d_{\Z}(x,y)=\beta^{\min( s(x,y), s_-(x,y))}$. Geometrically, this
distance is interpreted as follows: when one returns to the basis,
there is uniform contraction along stable manifolds (corresponding to
the past), and uniform expansion along unstable manifolds. Two points
are close in the unstable direction if they remain close in the
future for a long time (distance $\beta^{s(x,y)}$), while they are
close in the stable direction if they have a long common past
(distance $\beta^{s_-(x,y)}$).

\begin{thm}
\label{thm_Lq_inv}
Let $(T_{\Z},X_{\Z},\mu_{\Z})$ be the natural extension of a Young
tower in which the return time function $\phi$ has a moment of order
$q$. This system satisfies a concentration inequality with moment
$2q-2$, i.e., there exists a constant $C>0$ such that, for any
$n\in\N$, for any function $K_{\Z}(x_0,\dotsc, x_{n-1})$ which is
separately Lipschitz for the distance $d_{\Z}$,
  \begin{multline*}
  \int \left|K_{\Z}(x,\dotsc, T^{n-1}x) -\int K_{\Z}(y,\dotsc, T^{n-1}y)\dd\mu_{\Z}(y)\right|^{2q-2}
  \dd\mu_{\Z}(x)
  \\
  \leq C\left(\sum_j \Lip_j(K_{\Z})^2 \right)^{q-1}.
  \end{multline*}
\end{thm}
This implies Theorem~\ref{thm_nonuniform} (if one considers a
function $K_{\Z}$ depending only on the future of the points), but
the converse is not true: since the contraction is not uniform, we
are not able to reduce this theorem to Theorem~\ref{thm_nonuniform},
contrary to what we have done for subshifts of finite type or uniform
Young towers.

For the proof, we will work with the non-invertible system $X$, or
rather with an augmented space $X_*=X \cup\{x_*\}$ where $x_*$ is a
new point (at distance $1$ of any point of $X$, with zero measure).

Let us start with a function $K_{\Z}$ on $X_{\Z}$, depending on the
past and the future of points. We define a new function $K$ on
$X_*^n$ as follows. We let $K(x_0,\dotsc, x_{n-1})=K_{\Z}(y_0,\dotsc,
y_{n-1})$ where the $y_i$ are defined inductively. For each element
$a$ of the partition, let us fix an admissible past $p(a)$. Let us
also fix a point $y_*\in X_{\Z}$. Let $y_0 = (p((x_0)_0), x_0)$
(unless $x_0=x_*$, in which case let $y_0 = y_*$). If $y_{i-1}$ is
defined, let us define $y_i$. If $x_i=x_*$, we take $y_i=y_*$. If the
transition from $(x_{i-1})_0$ to $(x_i)_0$ is not permitted, let
$y_i=(p((x_i)_0), x_i)$. Otherwise, let $y_i=((y_{i-1})_{-\infty}^0,
x_i)$.

We claim that this function $K$ satisfies an inequality
  \begin{equation}
  \label{eq_K_inv}
  \int_{X_*} \left|K(x,\dotsc,T^{n-1}x) -\int K(y,\dotsc, T^{n-1}y)\dd \mu(y)\right|^{2q-2} \dd \mu(x)
  \leq C\left(\sum_{j=0}^{n-1} \Lip_j(K_{\Z})^2\right)^{q-1}.
  \end{equation}
This implies Theorem~\ref{thm_Lq_inv} by using the same argument as
in Subsection~\ref{subsec_bilateral}: let $K_N(y_0\breakdots
y_{n+N-1})=K_{\Z}(y_N,\dotsc, y_{N+n-1})$, and let $\tilde K_N$ be
the function obtained from $K_N$ by applying the above procedure.
After a change of variables, we get from~\eqref{eq_K_inv}
  \begin{equation*}
  \int_{X_{\Z}} \left| \tilde K_N(T^{-N} x,\dotsc, x, Tx,\dotsc, T^{n-1}x) - \E(\tilde K_N)\right|^{2q-2} \dd\mu_{\Z}(x)
  \leq C\left(\sum_{j=0}^{n-1} \Lip_j(K_{\Z})^2\right)^{q-1}.
  \end{equation*}
When $N$ tends to $\infty$, $\tilde K_N(T^{-N} x,\dotsc, x,
Tx,\dotsc, T^{n-1}x)$ converges to $K_{\Z}(x,\dotsc, T^{N-1}x)$.
Hence, we obtain the desired concentration inequality by letting $N$
tend to infinity in the previous equation.

To prove~\eqref{eq_K_inv}, we follow the same strategy as in the
previous section. Note that we can not directly apply
Theorem~\ref{thm_nonuniform} since the Lipschitz constants of $K$ are
not easily bounded in terms of those of $K_{\Z}$, due to the
non-uniform expansion. Therefore, we have to reimplement the strategy
from scratch.

Let us first start with a crucial remark. When one controls the
Lipschitz constants of $K$ in terms of those of $K_{\Z}$, a point
$x_*$ blocks the propagation of modifications, in the following
sense: consider a difference $K(x_0,\dotsc,x_{n-1}) - K(x'_0,\dotsc,
x'_{n-1})$ where $x_i$ and $x'_i$ coincide at all indices but $j$. By
construction of $K$, this is equal to $K_{\Z}(y_0,\dotsc, y_{n-1}) -
K_{\Z}(y'_0,\dotsc, y'_{n-1})$ for some points $y_i, y'_i \in
X_{\Z}$. The definition shows that $y_i=y'_i$ for $i< j$. On the
other hand, $y_i$ and $y'_i$ might be different for all $i\geq j$,
not only for $i=j$. However, if there is an index $k>j$ such that
$x_k=x'_k=x_*$, then $y_i=y'_i$ for $i\geq k$: this follows directly
from the construction. Therefore, $K(x_0,\dotsc,x_{n-1}) -
K(x'_0,\dotsc, x'_{n-1})$ will be estimated only in terms of
$\Lip_i(K_{\Z})$ for $j\leq i <k$.

To follow the same strategy as in the previous sections, we need to
show that $K_p$ is close to an integral, as in
Lemma~\ref{lem_controle_Kp_fullshift}. To do so, as in the proof of
this lemma, we define a function $f_i$ as in~\eqref{eq_define_fi},
and control its iterates under the transfer operator. We decompose
$K_p(x_p,\dotsc)=\sum_{i=0}^{p-1} \hatT^{p-i} f_i (x_p)+K(x_*,\dotsc,
x_*, x_p,\dotsc)$, where
  \begin{equation*}
  \begin{split}
  f_i(z)&=\sum_{T^i y = z}g^{(i)}(y)(K(y,\dotsc, T^i y, x_*,\dotsc,x_*, x_p,\dotsc)
  \\& \hphantom{= \sum_{T^i y = z}g^{(i)}(y)(}
  - K(y,\dotsc, T^{i-1}y, x_*,\dotsc, x_*, x_p,\dotsc))
  \end{split}
  \end{equation*}

When  $i < p-1$, there is a point $x_*$ in the definition of $f_i$,
blocking the propagation of modifications as we explained above.
Therefore, we may follow the proofs of Lemmas~\ref{lem_regularite_fi}
and~\ref{lem_Lrfi} in this setting, to obtain the following:
\begin{lem}
If $i < p-1$, we have for any $r\geq 0$ and any $z\in \Delta_0$
  \begin{equation*}
  \left|\hatT^r f_i(z)-\int_{\Delta} f_i\right|\leq \sum_{j=0}^i \Lip_j(K_{\Z})
  \left(\sum_{k=0}^r c_k^{(q-2)}c_{i-j+r-k}^{(q-1)}\right).
  \end{equation*}
\end{lem}
On the other hand, there is no such blocking effect for $f_{p-1}$,
yielding a worse estimate. Indeed, in $f_{p-1}$, one considers
averages of terms of the form $K(y,\dotsc, T^{p-1}y,
x_p,\dotsc)-K(y,\dotsc, T^{p-2}y, x_*, x_p,\dotsc)$. Considering the
definition of $K$ in terms of $K_{\Z}$, this difference reads
$K_{\Z}(y'_0,\dotsc, y'_{n-1}) - K_{\Z}(y''_0,\dotsc, y''_{n-1})$
where the points $y'_j, y''_j$ belong to $X_{\Z}$, coincide for $j <
p-1$ and may differ for $j\geq p-1$. For $j>p-1$, the points $y'_j$
and $y''_j$ have the same future, and the same past up to the index
$j-p$. Therefore, $d_{\Z}(y'_j, y''_j) \leq \beta^{\Card\{k\in [p, j]
\st x_k\in \Delta_0\}}$. Averaging over the points $y$ with
$T^{p-1}(y)=z$, we get
  \begin{equation*}
  |f_{p-1}(z)| \leq \sum_{j=p-1}^{n-1} \Lip_j(K_{\Z})
  \beta^{\Card\{k\in [p,j]\st x_k\in \Delta_0\}}.
  \end{equation*}
The functions $\hatT f_{p-1}$ and $\hatT f_{p-1}-\int f_{p-1}$ also
satisfy the same bound.

Still following the strategy of proof of
Section~\ref{sec_non_uniform}, we deduce from those estimates an
analogue of Lemma~\ref{control_Kp_non_uniform}, with an additional
error term coming from $f_{p-1}$: for all $x_p\in \Delta_0$,
  \begin{multline*}
  \left|K_p(x_p,\dotsc) -\int K(y,\dotsc, T^{p-1}y, x_p,\dotsc)\dd \mu(y)\right|
  \\
  \leq C\sum_{j=0}^{p-1} \Lip_j(K_{\Z}) c_{p-j}^{(q-2)} + C\sum_{j=p}^{n-1} \Lip_j(K_{\Z})
  \beta^{\Card\{k\in [p,j]\st x_k\in \Delta_0\}}.
  \end{multline*}

In turn, this yields an analogue of
Lemma~\ref{lem_devroye_Kp_opt_Lq}, still with an additional error
term: for all $\kappa\leq 2q$, and for all $x_{p+1}\in \Delta_0$
  \begin{align}
  \label{eq_conditionnelle_inversible}
  \raisetag{-60pt}
  \begin{split}
  \E( |D_p|^\kappa |\boF_{p+1})(x_{p+1},\dotsc)
  \leq {} & C\left(\sum_{j\geq p+1} \Lip_j(K_{\Z})
  \beta^{\Card\{k\in [p+1,j]\st x_k\in \Delta_0\}}\right)^\kappa
  \\ & \hspace{-21pt} + C\sum_{j=0}^p
  \Lip_j(K_{\Z})^\kappa c_{p-j}^{(q-2)} + C \sum_{h\geq 0}c_h^{(q-\kappa/2)}
  \left(\sum_{j=p-h+1}^{p}\Lip_j(K_{\Z})^2\right)^{\kappa/2}.
  \end{split}
  \end{align}
On the other hand, $\E( |D_p|^\kappa |\boF_{p+1})(x_{p+1},\dotsc) =
0$ if $h(x_{p+1})>0$.

We can now conclude the proof of~\eqref{eq_K_inv}, following the
strategy we used to prove Theorem~\ref{thm_nonuniform} in
Subsection~\ref{subsec_proof_nonuniform}. By Rosenthal-Burkholder
inequality, we have
  \begin{equation*}
  \E|K-\E K|^{2q-2} = \E\left|\sum D_p\right|^{2q-2}
  \leq C \E\left( \left[\sum_p \E(D_p^2 | \boF_{p+1})\right]^{q-1}\right)
  +C \sum \E( |D_p|^{2q-2}).
  \end{equation*}
The conditional expectations are estimated thanks
to~\eqref{eq_conditionnelle_inversible}. The terms that were already
present in the proof of Theorem~\ref{thm_nonuniform} are handled
exactly in the same way. Therefore, we only need to deal with the
additional term. Let us define a function $\Phi_j(x)=\beta^{\Card
\{k\in [1, j] \st T^k(x)\in \Delta_0\}}$ for $x\in \Delta_0$, and
$\Phi_j(x)=0$ elsewhere (it is closely related to the function
$\Psi_j$ of Lemma~\ref{lem_Psin}, with the difference that it is
supported in $\Delta_0$). The additional term in the
Rosenthal-Burkholder inequality is bounded by
  \begin{multline*}
  C\int \left[ \sum_{p\geq 0} \left(\sum_{j\geq p+1} \Lip_j(K_{\Z}) \Phi_{j-p-1}(T^{p+1}x)\right)^2 \right]^{q-1}
  \dd \mu(x)
  \\
  + C \sum_{p\geq 0} \int \left(\sum_{j\geq p+1} \Lip_j(K_{\Z}) \Phi_{j-p-1}(T^{p+1}x)\right)^{2q-2} \dd\mu(x).
  \end{multline*}
The inequality $\sum x_i^{q-1} \leq (\sum x_i)^{q-1}$ shows that the
second term is bounded by the first one. Therefore, to conclude the
proof, it is sufficient to prove the following inequality:
  \begin{equation}
  \label{eq:maineq_invert}
  \int \left[ \sum_{p\geq 0} \left(\sum_{j\geq p+1} \Lip_j(K_{\Z}) \Phi_{j-p-1}(T^{p+1}x)\right)^2 \right]^{q-1}
  \dd \mu(x)
  \leq C \left(\sum  \Lip_j(K_{\Z})^2\right)^{q-1}.
  \end{equation}

This estimate is formulated solely in terms of the non-invertible
system. Its proof is technical and complicated. Therefore, we defer
it to Theorem~\ref{thm_most_technical} in
Appendix~\ref{sec_appendix}. Modulo this result, this concludes the
proof of~\eqref{eq_K_inv}, and of Theorem~\ref{thm_Lq_inv}.

\section{Weak polynomial concentration inequalities}
\label{sec_non_uniform_weak}

The results of Section~\ref{sec_non_uniform} are not completely
satisfactory for the significant example of intermittent maps.
Indeed, for Pomeau-Manneville maps of index $\alpha\in (0,1)$ (with
$T(x)=x+cx^{1+\alpha}(1+o(1))$ for small $x$, see~\eqref{MPmap}
below), the return time function to the rightmost interval satisfies
a bound $\mu\{\phi = n\} \sim C/n^{1/\alpha+1}$. Therefore, the
corresponding Young tower has a moment of order $q$ for any
$q<1/\alpha$ (which yields a concentration inequality of order $Q$
for any $Q< 2/\alpha - 2$ when $\alpha < 1/2$), but it does not have
a moment of order $1/\alpha$. Indeed, it only has a \emph{weak}
moment of order $1/\alpha$, meaning that $\mu\{\phi > t\} \leq C
t^{-1/\alpha}$. An optimal concentration statement for such a map
would therefore be formulated in terms of weak moments. This is our
goal in this section.

\begin{thm}
\label{thm_nonuniform_weak}
Let $T:\Delta\to \Delta$ be a non-uniform Young tower. Assume that,
for some $q> 2$, the return time $\phi$ to the basis of the tower has
a weak moment of order $q$, i.e., there exists a constant $C>0$ such that
$\mu\{ x\in \Delta_0 \st \phi(x)>t\} \leq C t^{-q}$ for all $t>0$. Then $T$ satisfies
a weak polynomial concentration inequality with moment $2q-2$, i.e.,
there exists a constant $C>0$ such that, for any $n\in\N$, for any
separately Lipschitz function $K(x_0,\dotsc, x_{n-1})$, and any
$t>0$,
  \begin{multline*}
  \mu\left\{ x\st
  \left|K(x,\dotsc, T^{n-1}x) -\int K(y,\dotsc, T^{n-1}y)\dd\mu(y)\right| > t
  \right\}
  \\ \leq
  Ct^{-(2q-2)}\left(\sum_j \Lip_j(K)^2 \right)^{q-1}.
  \end{multline*}
\end{thm}

Let us introduce a convenient notation. When $Z$ is a real-valued
random variable and $Q\geq 1$, we write $\norm{Z}_{L^{Q, w}} = \sup t
P (|Z|>t)^{1/Q}$, so that $\Proba(|Z|>t) \leq t^{-Q}
\norm{Z}^Q_{L^{Q, w}}$. This is the weak $L^Q$ (semi)norm of $Z$.
With this notation, the statement of the theorem becomes
$\norm{K-\E(K)}_{L^{2q-2, w}}^{2q-2} \leq C \left(\sum_j \Lip_j(K)^2
\right)^{q-1}$, in close analogy with the statement of
Theorem~\ref{thm_nonuniform}. Note that $\norm{Z}_{L^{Q, w}}$ is not
a true norm: the triangle inequality fails, and is replaced by
$\norm{Z+Z'}_{L^{Q, w}} \leq C(
\norm{Z}_{L^{Q,w}}+\norm{Z'}_{L^{Q,w}})$. On the other hand,
  \begin{equation*}
  \norm{\max(|Z|, |Z'|)}_{L^{Q,w}}^Q \leq \norm{Z}_{L^{Q,w}}^Q+\norm{Z'}_{L^{Q,w}}^Q.
  \end{equation*}

Since a sequence with a weak moment of order $q>2$ has a strong
moment of order $2$, we may use intermediate results of the proof of
Theorem~\ref{thm_nonuniform} (and especially
Lemma~\ref{control_Kp_non_uniform}) to prove
Theorem~\ref{thm_nonuniform_weak}. The proofs diverge at the level of
Lemma~\ref{lem_devroye_Kp_opt_Lq}: the version we will need in the
weak moments case is the following.

\begin{lem}
\label{lem_devroye_Kp_opt_weak_Lq}
Assume that $\phi$ has a weak moment of order $q>2$. For all $t>0$,
  \begin{multline*}
  \Proba( |D_p|>t |\boF_{p+1})(x_{p+1}, \dotsc)
  \leq  Ct^{-(2q-2)} \sum_{j = 0}^p \Lip_j(K)^{2q-2}
  c_{p-j}^{(0)}
  \\
  + Ct^{-(2q-2)} \left(\sum \Lip_j(K)^2\right)^{q-2} \sup_{h>0} \left(h^{-1}\sum_{j=p-h+1}^p \Lip_j(K)\right)^2.
  \end{multline*}
\end{lem}
\begin{proof}
If $h(x_{p+1})>0$, then $x_{p+1}$ has a unique preimage $x_p$, and
$D_p(x_p, x_{p+1},\dotsc)=0$. Therefore, there is nothing to prove.
Assume now that $h(x_{p+1})=0$, and let $\{z_\alpha\}$ denote the
preimages of $x_{p+1}$ under $T$. Writing
$A(z)=D_p(z,x_{p+1},\dotsc)$, we have
  \begin{equation*}
  \Proba(|D_p|>t |\boF_{p+1})(x_{p+1}, \dotsc)=\sum_{|A(z_\alpha)|>t} g(z_\alpha).
  \end{equation*}

Since $\phi$ has a weak moment of order $q>2$, it has a strong moment
of order $2$. Therefore, \eqref{eq_mainineq_A} gives
  \begin{equation*}
  |A(z)|\leq \sum_{j\leq p-h} \Lip_j(K) c_{p-h-j}^{(0)} + \sum_{j=p-h+1}^p \Lip_j(K)
  \eqqcolon A_1(z)+A_2(z).
  \end{equation*}
If $|A(z)|>t$, then $A_1(z)>t/2$ or $A_2(z)>t/2$. Therefore,
$\Proba(|D_p|>t |\boF_{p+1})$ is bounded by
  \begin{equation}
  \label{wpxcouvxcvxcv}
  \sum_{A_1(z_\alpha)>t/2} g(z_\alpha) + \sum_{A_2(z_\alpha)>t/2}g(z_\alpha).
  \end{equation}
For the first sum,
  \begin{align*}
  \sum_{A_1(z_\alpha)>t/2} g(z_\alpha)
  &\leq C \sum g(z_\alpha) (A_1(z_\alpha)/t)^{2q-2}
  \\&
  \leq C \sum_{h\geq 0} \mu(\phi=h) t^{-(2q-2)} \left( \sum_{j\leq p-h} \Lip_j(K) c_{p-h-j}^{(0)} \right)^{2q-2}
  \\&
  \leq C t^{-(2q-2)} \sum_{h\geq 0} \mu(\phi=h) \sum_{j\leq p-h} \Lip_j(K)^{2q-2} c_{p-h-j}^{(0)}.
  \end{align*}
The coefficient of $\Lip_j(K)^{2q-2}$ in this expression is
$\sum_{h=0}^{p-j} c_h^{(2)} c_{p-h-j}^{(0)} \leq c_{p-j}^{(0)}$.
Therefore, this is bounded by $Ct^{-(2q-2)} \sum_{j\leq p}
\Lip_j(K)^{2q-2} c_{p-j}^{(0)}$.

The second sum of~\eqref{wpxcouvxcvxcv} is bounded by $C\sum
\mu(\phi=\ell)$, where the sum is restricted to those $\ell$ with
$\sum_{p-\ell+1}^p \Lip_j(K)>t/2$. Let $h$ be the smallest such
$\ell$, the sum is bounded by
  \begin{equation*}
  \mu(\phi \geq h) \leq C h^{-q} \leq C h^{-q} \left(\sum_{p-h+1}^p \Lip_j(K)/t\right)^{2q-2}.
  \end{equation*}
To bound the last sum, we use the inequality $(\sum_{p-h+1}^p x_j)^2
\leq h\sum x_j^2$, to obtain
  \begin{align*}
  h^{-q}\left(\sum_{p-h+1}^p \Lip_j(K)\right)^{2q-2}
  &= h^{-q}\left(\sum_{p-h+1}^p \Lip_j(K)\right)^{2} \cdot \left(\sum_{p-h+1}^p \Lip_j(K)\right)^{2q-4}
  \\&
  \leq h^{-q}\left(\sum_{p-h+1}^p \Lip_j(K)\right)^{2} \cdot \left(h\sum_{p-h+1}^p \Lip_j(K)^2\right)^{q-2}
  \\&
  \leq h^{-2} \left(\sum_{p-h+1}^p \Lip_j(K)\right)^{2} \cdot \left( \sum_{j\in \Z} \Lip_j(K)^2\right)^{q-2}.
  \end{align*}
This concludes the proof.
\end{proof}

To proceed, we need an analogue of Rosenthal-Burkholder inequality
for weak moments. Although it is not written explicitly in
Burkholder's article \cite{burkholder}, it follows easily from the
techniques developed there, giving the following statement.
\begin{thm}
\label{thm_rosenthal_weak}
Let $(D_p)$ be a sequence of reverse martingale differences with
respect to a decreasing filtration $\boF_p$ (i.e., $D_p$ is
$\boF_p$-measurable and $\E(D_p|\boF_{p+1}) = 0$). For all $Q\geq 2$,
  \begin{equation*}
  \norm{\sum D_p}_{L^{Q,w}}^Q \leq C \norm{\sum \E(D_p^2 |\boF_{p+1})}_{L^{Q/2,w}}^{Q/2}
  + C \norm{\sup |D_p|}_{L^{Q,w}}^Q.
  \end{equation*}
In particular,
   \begin{equation*}
  \norm{\sum D_p}_{L^{Q,w}}^Q \leq C \norm{\sum \E(D_p^2 |\boF_{p+1})}_{L^{Q/2,w}}^{Q/2}
  + C \sum \norm{D_p}_{L^{Q,w}}^Q.
  \end{equation*}
\end{thm}
\begin{proof}
By a truncation argument, it suffices to prove the result for bounded
random variables, and $p\in [0, P]$. Define three random variables
  \begin{equation*}
  X = \sup_{0\leq p \leq P} \left| \sum_{k=p}^P D_k\right|,
  \quad
  Y = \left(\sum \E(D_p^2 |\boF_{p+1})\right)^{1/2},
  \quad
  Z = \max_{0\leq p \leq P} |D_p|.
  \end{equation*}
The inequality (21.2) in \cite{burkholder} gives, for any
$0<\delta<\beta-1$,
  \begin{equation*}
  \Proba(X > \beta t, \max(Y,Z) \leq \delta t) \leq \epsilon \Proba(X > t),
  \end{equation*}
where $\epsilon=\delta^2/(\beta-\delta-1)^2$. In particular,
  \begin{align*}
  (\beta t)^{Q} \Proba(X > \beta t)
  &\leq (\beta t)^{Q} \Proba (\max(Y,Z) > \delta t) + (\beta t)^{Q} \epsilon \Proba(X > t)
  \\&
  \leq \beta^{Q}\delta^{-Q}  \norm{\max(Y,Z)}_{L^{Q,w}}^Q + \beta^{Q} \epsilon \norm{X}_{L^{Q,w}}^Q.
  \end{align*}
Taking the supremum over $t$, we obtain
  \begin{equation*}
  \norm{ X}_{L^{Q,w}}^Q \leq \beta^{Q}\delta^{-Q}  \norm{\max(Y,Z)}_{L^{Q,w}}^Q +
  \beta^{Q} \epsilon \norm{X}_{L^{Q,w}}^Q.
  \end{equation*}
If $\beta>1$ is fixed, and $\delta$ is chosen small enough so that
$\beta^Q \epsilon<1$, this yields $\norm{ X}_{L^{Q,w}}^Q \leq C
\norm{\max(Y,Z)}_{L^{Q,w}}^Q$. Since $\left|\sum_0^P D_p\right| \leq
X$ and $\norm{Y}_{L^{Q,w}}^Q = \norm{Y^2}_{L^{Q/2,w}}^{Q/2}$, this
proves the theorem.
\end{proof}

\begin{proof}[Proof of Theorem~\ref{thm_nonuniform_weak}]
We have $K-\E(K)= \sum D_p$, hence
  \begin{equation*}
  \norm{K-\E(K)}_{L^{2q-2, w}}^{2q-2}
  \leq C \norm{ \sum \E(D_p^2 |\boF_{p+1})}_{L^{q-1,w}}^{q-1}
  + C \sum \norm{D_p}_{L^{2q-2,w}}^{2q-2}.
  \end{equation*}
For the first term, we use the inequality $\norm{\cdot}_{L^{Q,w}}
\leq \norm{\cdot}_{L^Q}$. Therefore, this term is bounded by
  \begin{equation*}
  C \E\left( \left[\sum_p \E(D_p^2 | \boF_{p+1})\right]^{q-1}\right).
  \end{equation*}
Since $\phi$ has a weak moment of order $q$, it has a strong moment
of order $2$. Therefore, \eqref{pouwxcvlmkj} gives
$\E(D_p^2|\boF_{p+1}) \leq \sum_{j\leq p} c_{p-j}^{(0)} \Lip_j(K)^2$.
Hence, the first term in Rosenthal-Burkholder inequality is bounded
by
  \begin{equation*}
  C\left(\sum_{p} \sum_{j=0}^p \Lip_j(K)^2 c_{p-j}^{(0)}\right)^{q-1}
  \leq C \left(\sum_j \Lip_j(K)^2 \right)^{q-1}.
  \end{equation*}

Let us now turn to $\norm{D_p}_{L^{2q-2,w}}$. Integrating the
estimates of Lemma~\ref{lem_devroye_Kp_opt_weak_Lq}, we get
  \begin{equation}
  \label{omiw<uxcvmlkj}
  \norm{D_p}_{L^{2q-2, w}}^{2q-2}
  \leq C \sum_{j\leq p} \Lip_j(K)^{2q-2}
  c_{p-j}^{(0)}
  + C \left(\sum \Lip_j(K)^2\right)^{q-2} \sup_{h>0} \left(h^{-1}\sum_{j=p-h+1}^p \Lip_j(K)\right)^2.
  \end{equation}
We should sum those estimates over $p$. For the first sum, we obtain
  \begin{equation*}
  \sum_j \Lip_j(K)^{2q-2} \sum_{p\geq j} c_{p-j}^{(0)}
  \leq C \sum_j \Lip_j(K)^{2q-2} \leq C \left(\sum_j \Lip_j(K)^2 \right)^{q-1}.
  \end{equation*}
For the second sum, let us define a function $f$ on $\Z$ by $f(j) =
\Lip_j(K)$. This function belongs to $\ell^2(\Z)$. The corresponding
maximal function $Mf(p)= \sup_{h>0} \frac{1}{2h+1} \sum_{j=p-h}^{p+h}
f(j)$ also belongs to $\ell^2(\Z)$ and satisfies $\norm{Mf}_{\ell^2}
\leq C \norm{f}_{\ell^2}$, by Hardy-Littlewood maximal inequality. In
particular,
  \begin{equation*}
  \sum_p \sup_{h>0} \left(h^{-1}\sum_{p-h+1}^p \Lip_j(K)\right)^2
  \leq C \sum_j \Lip_j(K)^2.
  \end{equation*}
Therefore, the contribution of the second term
in~\eqref{omiw<uxcvmlkj} is bounded by $C\left(\sum
\Lip_j(K)^2\right)^{q-1}$. This concludes the proof of
Theorem~\ref{thm_nonuniform_weak}.
\end{proof}

\begin{rmk}
In view of Theorems~\ref{thm_Lq_inv} and~\ref{thm_nonuniform_weak},
it would seem natural to try to prove a weak polynomial concentration
inequality in invertible systems with weak moment controls on the
return time. We have not been able to prove such a statement.
\end{rmk}

\section{Applications}

\label{sec_applications}

In this section, we first give examples of dynamical systems
satisfying an exponential concentration inequality or only a
polynomial concentration inequality. We also give examples of systems
satisfying a weak polynomial concentration inequality. Second, we
present several applications of these inequalities to specific
observables. We shall not attempt to be exhaustive. Previous results
are found in
\cite{collet_concentration_BV,chazottes_collet_schmitt,chazottes_intermittent}.
For instance, we strengthen the bounds obtained in
\cite{chazottes_collet_schmitt} since for dynamical systems modeled
by a uniform Young tower with exponential tails, we can now use an
exponential concentration inequality instead of a polynomial
concentration inequality with moment $2$ as in
\cite{chazottes_collet_schmitt}. For systems modeled by a non-uniform
Young tower, only a polynomial concentration inequality with moment
$2$ was known for intermittent maps of the interval (under some
restrictions on the parameter). We now have at our disposal an
optimal polynomial concentration inequality for these maps, and more
generally, for dynamical systems modeled by non-uniform Young towers
with polynomial tails.

\subsection{Examples of dynamical systems}

There are well-known dynamical systems $(X,T)$ which can be modeled
by a uniform Young tower with exponential tails
\cite{lsyoung_annals}. Examples of invertible dynamical systems
fitting this framework are for instance Axiom A attractors, H\'enon
attractors for Benedicks-Carleson parameters
\cite{benedicks_lsyoung}, piecewise hyperbolic maps like the Lozi
attractor, some billiards with convex scatterers, etc. Such systems
admit an SRB measure $\mu$ and there is an invertible uniform Young
tower $(\Delta_\Z,\hat{T}_\Z,\hat{\mu}_\Z)$ and a projection map
$\pi:\Delta_\Z\to X$ such that $T\circ \pi=\pi \circ \hat{T}_\Z$ and
$\mu=\hat{\mu}_\Z\circ \pi^{-1}$. In the non-invertible case, there
is a non-invertible Young tower $(\Delta,\hat{T},\hat{\mu})$ and a
corresponding projection map. A non-invertible example is the
quadratic family for Benedicks-Carleson parameters. In both cases, it
can also be ensured that the projection map is contracting, i.e.,
$d(\pi x, \pi y) \leq \hat d_\beta(x,y)$ for every $x,y$ in the same
partition element. Here, $\hat d_\beta$ denotes the (unilateral or
bilateral) symbolic distance in the tower given by $\hat d_\beta(x,y)
= \beta^{s(x,y)}$ for some $\beta<1$. In particular, if $f$ is a
bounded Lipschitz function on $X$, it lifts to a function $f\circ
\pi$ which is Lipschitz in the tower. More generally, if $f$ is
H\"older continuous, then its lift is Lipschitz for $\hat d_\beta$ if
$\beta$ is close enough to $1$. Therefore, all the results we proved
in the previous sections for Lipschitz observables $K$ have a
counterpart about H\"older ones, we will not give further details in
this direction and restrict to the Lipschitz situation for ease of
exposition. We will also assume for simplicity that $X$ is bounded.

\begin{thm}
Let $(X,T)$ be a dynamical system modeled by a uniform Young tower
with exponential tails and let $\mu$ be its SRB measure. There exists
$C>0$ such that, for any $n\in \N$, for any separately Lipschitz
function $K(x_0,\dotsc, x_{n-1})$,
\begin{equation}
\label{main_ineq_bis}
\int e^{K(x,Tx,\dotsc, T^{n-1}x)} \dd \mu (x) \leq e^{\int K(x,\dotsc, T^{n-1}x)\dd \mu(x)}
e^{C\sum_{j=0}^{n-1} \Lip_j(K)^2}.
\end{equation}
\end{thm}

This theorem is an obvious consequence of
Theorem~\ref{thm_exponential_bis} in the invertible case and of
Theorem~\ref{thm_exponential} in the non-invertible case.
Inequality~\eqref{main_ineq_bis} was previously known only for
uniformly piecewise expanding maps of the interval and subshifts of
finite type equipped with a Gibbs measure
\cite{collet_concentration_BV}. Under the assumptions of the previous
theorem, only a polynomial concentration with moment $2$ had been
proven \cite{chazottes_collet_schmitt0}.

An immediate consequence of~\eqref{main_ineq_bis} is the following
inequality for upper deviations: for all $t>0$ and for all $n\in\N$
 \begin{multline}\label{dev_exp}
 \mu\left\{ x \in X\st K(x,Tx,\dotsc, T^{n-1}x) -\int K(y,\dotsc, T^{n-1}y)\dd \mu(y)>t\right\}\\
 \leq e^{-\frac{t^2}{4C\sum_{j=0}^{n-1} \Lip_j(K)^2}}.
 \end{multline}
The same bound holds for lower deviations by applying~\eqref{dev_exp}
to $-K$.

Let us now consider dynamical systems modeled by a non-uniform Young
tower with polynomial tails. In the invertible case, there is an
invertible non-uniform Young tower
$(\Delta_\Z,\hat{T}_\Z,\hat{\mu}_\Z)$ and a projection map
$\pi:\Delta_\Z\to X$, and the SRB measure is
$\mu=\hat{\mu}_\Z\circ\pi^{-1}$ provided that $\sum \phi(\alpha)
\hat{\mu}_\Z(\Delta_{\alpha,0}) <\infty$. If $\sum \phi(\alpha)^q
\hat{\mu}_\Z(\Delta_{\alpha,0}) <\infty$, we shall simply say that
the tower has $L^q$ tails. Similarly, if $\sum_{\phi(\alpha)>n}
\hat{\mu}_\Z(\Delta_{\alpha,0}) \leq C n^{-q}$, we shall say that the
tower has weak $L^q$ tails. We can of course rephrase what we have
just said in the non-invertible case.

\begin{thm}
\label{thm_nonuniform_bis}
Let $(X,T)$ be a dynamical system modeled by a non-uniform Young
tower with $L^q$ tails, for some $q\geq 2$. Then $T$ satisfies a
polynomial concentration inequality with moment $2q-2$, i.e., there
exists a constant $C>0$ such that, for any $n\in\N$, for any
separately Lipschitz function $K(x_0,\dotsc, x_{n-1})$,
  \begin{equation*}
  \int \left|K(x,\dotsc, T^{n-1}x) -\int K(y,\dotsc, T^{n-1}y)\dd\mu(y)\right|^{2q-2}\dd\mu(x)
  \leq C\left(\sum_{j=0}^{n-1} \Lip_j(K)^2 \right)^{q-1}.
  \end{equation*}
\end{thm}

Using Markov's inequality we get at once that, for any $t>0$ and for any $n\in\N$,
\begin{multline}\label{dev_poly}
 \mu\left\{ x \in X\st \big|K(x,Tx,\dotsc, T^{n-1}x) -\int K(y,\dotsc, T^{n-1}y)\dd \mu(y)\big|>t\right\}\\
 \leq C\, \frac{\left(\sum_{j=0}^{n-1} \Lip_j(K)^2 \right)^{q-1}}{t^{2q-2}}.
 \end{multline}

If the tails are only in weak $L^q$,
Theorem~\ref{thm_nonuniform_weak} shows that~\eqref{dev_poly} still
holds.

The fundamental example is an expanding map of the interval with an
indifferent fixed point \cite{lsyoung_recurrence}. For the sake of
definiteness, we consider for $\alpha\in(0,1)$ the so-called
``intermittent'' map $T:[0,1]\to[0,1]$ defined by
\begin{equation}\label{MPmap}
 T(x) =
  \begin{cases}
  x(1+2^\alpha x^\alpha) & \text{if } 0\leq x \leq 1/2, \\
   2x-1       & \text{if } 1/2<x \leq 1.
  \end{cases}
\end{equation}
There is a unique absolutely continuous invariant probability measure
$\dd\mu(x)=h(x)\dd x$ such that $h(x)\sim x^{-\alpha}$ as $x\to 0$.
This map is modeled by a non-uniform Young tower $(\Delta,\hat{\mu})$
such that $\hat{\mu}\{\phi=n\}\sim C/n^{\frac{1}{\alpha}+1}$. The
return time has a weak moment of order $1/\alpha$. Thus, for
$\alpha\in (0,1/2)$, the previous results yield:

\begin{cor}\label{cor-MP}
Let $T$ be the map~\eqref{MPmap} and $\mu$ its absolutely continuous
invariant probability measure. There exists a constant $C>0$ such
that, for any $n\in \N$, for any separately Lipschitz function
$K(x_0,\dotsc, x_{n-1})$,
\begin{multline*}
 \mu\left\{ x \in X\st \big|K(x,Tx,\dotsc, T^{n-1}x) -\int K(y,\dotsc, T^{n-1}y)\dd \mu(y)\big|>t\right\}\\
 \leq C\, \frac{\left(\sum_{j=0}^{n-1} \Lip_j(K)^2 \right)^{1/\alpha-1}}{t^{\frac{2}{\alpha}-2}}.
 \end{multline*}
\end{cor}

This estimate readily gives bounds for the moments of order $q\not=
2/\alpha-2$. Indeed, if $Z$ is a random variable satisfying
$\Proba(|Z|>t)\leq (A/t)^Q$, then using the formula $\E(|Z|^q) = \int
q t^{q-1}\Proba(|Z|>t)\dd t$ and the tail estimates, one gets
  \begin{equation*}
  \E(|Z|^q) \leq \frac{Q}{Q-q} A^q \quad\text{for }q<Q,
  \end{equation*}
and if $Z$ is bounded
  \begin{equation*}
  \E(|Z|^q) \leq \frac{q}{q-Q} A^Q \norm{Z}_{L^\infty}^{q-Q} \quad \text{for }q>Q.
  \end{equation*}


For $q<2/\alpha-2$, this generalizes to arbitrary separately
Lipschitz functions of $n$ variables the moment bounds obtained for
ergodic sums of Lipschitz functions in
\cite{melbourne_nicol_large_deviations} (while the moment bounds for
$q>2/\alpha-2$ are apparently new, even for ergodic sums). On the
other hand, we improve the result in \cite{chazottes_intermittent} in
two respects: first, we obtain a polynomial concentration inequality
with moment $2$ for any $\alpha\in(0,1/2)$ instead of
$(0,4-\sqrt{15})$; second, we also obtain a polynomial concentration
inequality with a moment whose order is larger than $2$ and depends
on $\alpha\in(0,1/2)$.

\begin{rmk}
There is a difference between Theorems~\ref{thm_nonuniform} (about
strong moments) and~\ref{thm_nonuniform_weak} (about weak moments):
in the former, the range of parameters is $q\geq 2$, while we require
$q>2$ in the latter. It turns out that
Theorem~\ref{thm_nonuniform_weak} is \emph{false} for $q=2$, as
testified by the map~\eqref{MPmap} with $\alpha=1/2$. For such a map,
if $f$ is a H\"older function with $\int f\dd\mu=0$ and $f(0)\not=0$,
then $S_n f/\sqrt{n\log n}$ converges in distribution to a gaussian
\cite[Page 88]{gouezel_stable}. If Theorem~\ref{thm_nonuniform_weak}
were true for $q=2$, we would have $\mu\{ |S_n f|>t\} \leq C t^{-2}
n$, hence $\mu\{ |S_n f / \sqrt{n \log n}|> t\} \leq C t^{-2}(n\log
n)^{-1} n \to 0$, implying that $S_n f / \sqrt{n \log n}$ tends in
probability to $0$ and giving a contradiction.
\end{rmk}

There are also invertible examples exhibiting an intermittent
behavior, notably coming from billiards. Indeed, apart from the
stadium billiard (with a weak moment of order $2$ and therefore not
covered by our results), Chernov and Zhang studied in
\cite{chernov_zhang_polynomial, chernov_zhang_variable} several
classes of billiards for which the decay of correlations behaves like
$O((\log n)^C / n^{1/\alpha-1})$, for some parameter $\alpha$ that
can be chosen freely in $(0,1/2]$ and some $C>0$. This decay rate is
obtained by modeling those billiards by nonuniform invertible Young
towers with well controlled tails. Therefore, we can apply
Theorem~\ref{thm_nonuniform_bis} to those maps, yielding polynomial
concentration inequalities for any exponent $p<2/\alpha-2$, just like
in the above one-dimensional non-invertible situation.

\subsection{Empirical covariance}

For a Lipschitz observable $f$ such that $\int f\dd\mu=0$, the auto-covariance of the process $\{f\circ T^k\}$
is defined as usual by
\begin{equation}\label{def-autocov}
C_f(\ell)=\int f\cdot f\circ T^\ell \dd\mu.
\end{equation}
An obvious estimator for $C_f(\ell)$ is
\[
\widehat{C}_f(n,\ell,x)=\frac{1}{n} \sum_{j=0}^{n-1} f(T^j x)f(T^{j+\ell} x).
\]
We could as well consider the covariance between $\{f\circ T^k\}$ and
$\{g\circ T^k\}$, for a pair of Lipschitz observables $f,g$. For each
$\ell\geq 0$, the ergodic theorem tells us that
$\widehat{C}_f(n,\ell,x)\to C_f(\ell)$ $\mu$-almost surely, as
$n\to\infty$. Considering the function of $n+\ell$ variables
$K(x_0,\dotsc,x_{n+\ell-1})=\frac{1}{n} \sum_{j=0}^{n-1}
f(x_j)f(x_{j+\ell})$, we obtain immediately (noting that $\int
\widehat{C}_f(n,\ell,x)\dd\mu(x)= C_f(\ell)$) the following theorems.

\begin{thm}
Let $(X,T)$ be a dynamical system modeled by a uniform Young tower
with exponential tails and $\mu$ its SRB measure. Let $f:X\to\R$ be a
Lipschitz function. There exists a constant $c>0$ such that, for any
$n,\ell\in\N$ and for any $t>0$,
\[
\mu\left\{x \in X\st  \big|\widehat{C}_f(n,\ell,x)-C_f(\ell)\big| >t\right\}
\leq 2 e^{-c \frac{n^2t^2}{n+\ell}}.
\]
\end{thm}

\begin{thm}
Let $(X,T)$ be a dynamical system modeled by a non-uniform Young
tower with weak $L^q$ tails, for some $q\geq 2$, and $\mu$ its SRB
measure. Let $f:X\to\R$ be a Lipschitz function. There exists a
constant $c>0$ such that, for any $n,\ell\in\N$ and for any $t>0$,
\[
\mu\left\{x\in X :  \big|\widehat{C}_f(n,\ell,x)-C_f(\ell)\big| >t\right\}
\leq c\ \Big(\frac{n+\ell}{n^2}\Big)^{q-1}\frac{1}{t^{2q-2}}.
\]
\end{thm}

\subsection{Empirical measure}

Given $x\in X$ in an ergodic compact dynamical system $(X,T,\mu)$,
let
\[
\boE_n(x)=\frac{1}{n} \sum_{j=0}^{n-1} \delta_{T^j x}
\]
be the associated empirical measure. By Birkhoff's ergodic theorem, $\boE_n(x)$ vaguely
converges to $\mu$, for $\mu$-almost every $x$. Our aim is to quantify the `speed' at which
this convergence takes place. We use the Kantorovich distance
(compatible with vague convergence): for two probability measures
$\mu_1,\mu_2$ on $X$, let
\begin{equation*}
\textup{dist}_{\scriptscriptstyle{K}}(\mu_1,\mu_2)=
\sup\left\{ \int g \dd\mu_1 -\int g \dd\mu_2 : g:X\to\R\text{ is $1$-Lipschitz}\right\}.
\end{equation*}
Set
\[
\mathcal{D}_n(x)=\textup{dist}_{\scriptscriptstyle{K}}(\boE_n(x),\mu).
\]
We have the following general bounds.

\begin{thm}\label{dev_kanto-exp}
Let $(X,T)$ be a dynamical system modeled by a uniform Young tower with exponential tails
and $\mu$ its SRB measure. Let $f:X\to\R$ be a Lipschitz function with $\int f\dd\mu=0$.
There exists a constant $C>0$ such that, for any $n\in\N$ and for any $t>0$,
\[
\mu\left\{x \in X\st
\Big|\mathcal{D}_n(x)-
\int \mathcal{D}_n(y)\dd\mu(y)\Big| >\frac{t}{\sqrt{n}}\right\}
\leq 2 e^{-C t^2}.
\]
\end{thm}

\begin{thm}\label{dev_kanto-poly}
Let $(X,T)$ be a dynamical system modeled by a non-uniform Young
tower with weak $L^q$ tails, for some $q\geq 2$, and $\mu$ its SRB
measure. Let $f:X\to\R$ be a Lipschitz function with $\int
f\dd\mu=0$. There exists a constant $C>0$ such that, for all $n\in\N$
and all $t>0$,
\[
\mu\left\{x \in X\st
\Big|\mathcal{D}_n(x)-
\int \mathcal{D}_n(y)\dd\mu(y)\Big| >\frac{t}{\sqrt{n}}\right\}
\leq \frac{C}{t^{2q-2}}.
\]
\end{thm}

These bounds follow at once by applying either~\eqref{dev_exp}
or~\eqref{dev_poly} to the function
\[
K(x_0,\dotsc,x_{n-1})=
\sup\left\{ \frac{1}{n}\sum_{j=0}^{n-1} g(x_j) -\int g \dd\mu : g:X\to\R\;\textup{is}\;1-\textup{Lipschitz}\right\}
\]
whose Lipschitz constants are uniformly bounded by $1/n$. The natural next step is to seek for
an upper bound for $\int \mathcal{D}_n(y)\dd\mu(y)$. We are not
able to obtain an {\em a priori} sufficiently good estimate unless we restrict to one-dimensional
systems.

\begin{cor}\label{vitesse_exp_kanto}
Let $(X,T)$ be a one-dimensional dynamical system satisfying the
assumptions of Theorem~\ref{dev_kanto-exp}. There exist some
constants $B,C>0$ such that, for any $n\in\N$ and for any $t>0$,
\[
\mu\left\{x \in X\st\mathcal{D}_n(x)> \frac{t}{n^{1/2}}+\frac{B}{n^{1/4}}\right\}\leq e^{-Ct^2}.
\]
\end{cor}

\begin{cor}\label{vitesse_poly_kanto}
Let $(X,T)$ be a one-dimensional dynamical system satisfying the
assumptions of Theorem~\ref{dev_kanto-poly}. There exist some
constants $B,C>0$ such that, for any $n\in\N$ and for any $t>0$,
\[
\mu\left\{x \in X\st\mathcal{D}_n(x)> \frac{t}{n^{1/2}}+\frac{B}{n^{1/4}}\right\}\leq \frac{C}{t^{2q-2}}.
\]
\end{cor}

These two corollaries follow immediately if we can prove that
there exists $B>0$ such that, for any $n\in\N$,
\[
\int \mathcal{D}_n\dd\mu\leq \frac{B}{n^{1/4}}.
\]
The proof is found in \cite[Theorem 5.2]{chazottes_collet_schmitt}.
The point is that in dimension one, there is a special representation
of Kantorovich distance in terms of the distribution functions. The
estimate then follows easily using the fact that the auto-covariance
of Lipschitz observables is summable under the above assumptions.

For the map~\eqref{MPmap}, we can use Corollary~\ref{cor-MP} to get
the bound
\[
\mu\left\{x \in X\st\mathcal{D}_n(x)>\frac{t}{n^{1/2}}+\frac{B}{n^{1/4}}\right\}\leq \frac{C}{t^{\frac{2}{\alpha}-2}},
\]
for any $n\in\N$ and for any $t>0$.

\begin{rmk}
What explains the power $1/4$ of $n$ is the fact that at some stage,
one has to approximate a characteristic function of a set by a
Lipschitz function. If one can control the auto-covariance of
functions with bounded variation, one gets
\[
\int \mathcal{D}_n\dd\mu\leq \frac{B}{\sqrt{n}}.
\]
This is the case for uniformly piecewise expanding maps of the interval \cite{collet_concentration_BV}. This is
also the case for the quadratic map with Benedicks-Carleson parameters \cite{lsyoung_quadratic}.
Since we proved that this system satisfies an exponential concentration inequality, we get
\[
\mu\left\{x \in X\st\mathcal{D}_n(x)>\frac{t}{\sqrt{n}}\right\}\leq e^{-C t^2},
\]
for any $n\in\N$ and for any $t$ greater than some $t_0>0$.
\end{rmk}

\subsection{Kernel density estimation}

The estimation from an orbit of the density $h$ of the invariant measure of a one-dimensional dynamical system $(X,T)$ is based on the estimator
\[
h_n(s;x)=\frac{1}{n a_n} \sum_{j=0}^{n-1} \psi\Big(\frac{s-T^j x}{a_n} \Big)
\]
where $a_n$ is a sequence of positive numbers going to $0$ but such
that $n a_n$ goes to $\infty$, and $\psi$ is a `kernel', that is, a
non-negative Lipschitz function with compact support. We suppose that
it is fixed in the sequel.

As proved in \cite[Appendix C]{chazottes_collet_schmitt0}, the density of the invariant measure
for a one-dimensional system modeled by a uniform Young tower with exponential tails has the following property:
there exist some constants $B>0$ and $\tau>0$ such that
\begin{equation}\label{besov}
\int \big| h(s)-h(s-t)\big| \dd s \leq B |t|^\tau,\;\forall t\in\R.
\end{equation}
We have the following result about the $L^1$ convergence of empirical densities.
\begin{thm}\label{kernel}
Let $(X,T)$ be a one-dimensional dynamical system modeled by a uniform Young tower with exponential tails
and $\mu$ its SRB measure. There exist $c_1,c_2>0$ such that, for
any $t>c_1(a_n^\tau +1/(\sqrt{n}a_n^2))$ and for any $n\in\N$
\[
\mu\left\{x \in X\st \int  \big| h_n(s;x)-h(s)\big| \dd s> t\right\} \leq e^{-c_2 n a_n^2 t^2}.
\]
\end{thm}
The proof is similar to the proof of Theorem 5.2 in
\cite{chazottes_collet_schmitt0} except that we use an exponential
concentration inequality instead of a polynomial concentration
inequality with moment $2$; hence we obtain a much stronger bound.
(See also \cite[Theorem III.2]{collet_concentration_BV} for uniformly
piecewise expanding maps of the interval.) The property~\eqref{besov}
is used to obtain an upper bound for $\int \big| h_n(s;x)-h(s)\big|
\dd s\dd\mu$.

We do not know if the property~\eqref{besov} holds for the density of
the invariant measure of all one-dimensional system modeled by a
non-uniform Young tower with polynomial tails. But for the special
case of the intermittent map~\eqref{MPmap}, it is easy to check
that~\eqref{besov} is true with $\tau=1-\alpha$. Therefore, applying
Corollary~\ref{cor-MP} we get the following result.

\begin{thm}
Let $T$ be the map~\eqref{MPmap} and $\mu$ its absolutely continuous
invariant probability measure. There exist $c_1,c_2>0$ such that for
any $t>c_1(a_n^{1-\alpha} +1/(\sqrt{n}a_n^2))$ and for any $n\in\N$
\[
\mu\left\{x \in X \st \int  \big| h_n(s;x)-h(s)\big| \dd s> t\right\} \leq
\frac{c_2}{n^{\frac{1}{\alpha}-1}a_n^{\frac{2}{\alpha}-2} t^{\frac{2}{\alpha}-2}}.
\]
\end{thm}

\subsection{Tracing orbit properties}

Let $A$ be a measurable subset of  $X$ such that $\mu(A)>0$ and define for all $n\in\N$
\begin{equation*}
\mathcal{S}_A(x,n)=\frac{1}{n} \inf_{y\in A} \sum_{j=0}^{n-1} d(T^j x,T^j y),
\end{equation*}
where $d$ is the distance on $X$. This quantity, between $0$ and $1$,
measures how well we  can trace the orbit of some initial condition not in $A$
by an orbit from an element of $A$.

\begin{thm}
Let $(X,T)$ be a dynamical system modeled by a uniform Young tower with exponential tails
and $\mu$ its SRB measure. There exist constants $c_1,c_2>0$ such that, for any
measurable subset $A\subset X$ with $\mu(A)>0$, for any $n\in\N$ and for any $t>0$
\[
\mu\left\{x \in X\st \mathcal{S}_A(x,n)> c_1\frac{\sqrt{\log n}}{\mu(A)\sqrt{n}} + \frac{t}{\sqrt{n}}\right\}
\leq e^{-c_2t^2}.
\]
\end{thm}

Again, the proof is the same as \cite[Theorem
IV.1]{collet_concentration_BV} because it relies only on the
exponential concentration inequality.

\begin{thm}
Let $(X,T)$ be a dynamical system modeled by a non-uniform Young
tower with weak $L^q$ tails, for some $q\geq 2$, and $\mu$ its SRB
measure. There exist constants $c_1,c_2>0$ such that, for any
measurable subset $A\subset X$ with $\mu(A)>0$, for any $n\in\N$ and
for any $t>0$
\[
\mu\left\{x \in X\st \mathcal{S}_A(x,n)>
\frac{1}{n^{(q-1)/(2q-1)}}\left(t+\frac{c_1}{\mu(A)}\right)\right\}
\leq \frac{c_2}{n^{(q-1)/(2q-1)}t^{2q-2}}.
\]
\end{thm}

The proof follows the lines of that of \cite[Theorem
IV.1]{collet_concentration_BV} except that one uses the weak
polynomial concentration inequality instead of the exponential
concentration inequality as in the previous theorem.

For the intermittent maps~\eqref{MPmap}, we can use
Corollary~\ref{cor-MP}. We get that there exist constants $c_1,c_2>0$
such that for any subset $A\subset [0,1]$ with $\mu(A)>0$, for any
$n\in\N$ and for any $t>0$
\[
\mu\left\{x \in [0,1]: \mathcal{S}_A(x,n)>
\frac{1}{n^{(1/\alpha-1)/(2/\alpha-1)}}\left(t+\frac{c_1}{\mu(A)}\right)\right\}
\leq \frac{c_2}{n^{(\frac{1}{\alpha}-1)/(\frac{2}{\alpha}-1)}t^{\frac{2}{\alpha}-2}}.
\]

We now formulate similar results for the number of mismatches at a given precision.
Let $A$ be a measurable subset of  $X$ such that $\mu(A)>0$ and $\epsilon>0$.
For all $n\in\N$ define
\begin{equation*}
\mathcal{M}_A(x,n,\epsilon)=\frac{1}{n} \inf_{y\in A} \Card\{0\leq j\leq n-1 \st d(T^j x,T^j j)>\epsilon\}.
\end{equation*}
We have the following result.
\begin{thm}
Let $(X,T)$ be a dynamical system modeled by a Young tower with exponential tails
and $\mu$ its SRB measure. There exist constants $c_1,c_2>0$ such that, if
$A\subset X$ is such that $\mu(A)>0$, then for any $0<\epsilon<1/2$, for any $n\in\N$ and for any $t>0$
\[
\mu\left\{x\in X : \mathcal{M}_A(x,n,\epsilon)>
c_1\epsilon^{-1}\frac{\sqrt{\log n}}{\mu(A)\sqrt{n}} + \frac{t\epsilon^{-1}}{\sqrt{n}}\right\}
\leq e^{-c_2t^2}.
\]
\end{thm}

\begin{thm}
Let $(X,T)$ be a dynamical system modeled by a non-uniform Young
tower with weak $L^q$ tails, for some $q\geq 2$, and $\mu$ its SRB
measure. There exist constants $c_1,c_2>0$ such that, if $A\subset X$
is such that $\mu(A)>0$, then for any $0<\epsilon<1/2$, for any
$n\in\N$ and for any $t>0$
\begin{multline*}
\mu\left\{x \in X\st \mathcal{M}_A(x,n,\epsilon)>
\frac{1}{\epsilon^{(q-1)/(q-1/2)}n^{(q-1)/(2q-1)}}\left(t+\frac{c_1}{\mu(A)}\right)\right\}\\
\leq \frac{c_2}{\epsilon^{(q-1)/(q-1/2)}n^{(q-1)/(2q-1)}t^{2q-2}}.
\end{multline*}
\end{thm}

Once more, the proofs are almost the same as \cite[Theorem
IV.2]{collet_concentration_BV}.

\subsection{Integrated periodogram}

Let $(X,T,\mu)$ be a dynamical system and $f:X\to\R$ be a Lipschitz
function such that $\int f\dd\mu=0$. Define the empirical integrated
periodogram function of the process $\{f\circ T^k\}_{k\geq 0}$ by
\[
J_n(x,\omega)=\int_0^\omega \frac{1}{n} \Big| \sum_{j=0}^{n-1} e^{-\ic js} f(T^j x)\Big|^2 \dd s,\quad \omega\in[0,2\pi].
\]
Let
\[
J(\omega)=C_f(0)\omega + 2\sum_{\ell=1}^\infty \frac{\sin(\omega\ell)}{\ell} \, C_f(\ell),
\]
where $C_f(\ell)$ is defined in~\eqref{def-autocov}.

\begin{thm}
Let $(X,T)$ be a dynamical system modeled by a uniform Young tower with exponential tails
and $\mu$ its SRB measure. Let $f:X\to\R$ be a Lipschitz function such that $\int f\dd\mu=0$.
There exist some positive constants $c_1,c_2$ such that
for any $n\in\N$ and for any $t>0$
\[
\mu\left\{x \in X\st \sup_{\omega\in[0,2\pi]}\big|J_n(x,\omega)-J(\omega)\big|> t +
\frac{c_1(1+\log n)^{3/2}}{\sqrt{n}}\right\}
\leq e^{-c_2 n t^2/(1+\log n)^2}.
\]
\end{thm}

The observable
$\sup_{\omega\in[0,2\pi]}\big|J_n(x,\omega)-J(\omega)\big|$ was
studied in \cite{chazottes_collet_schmitt} in the same setting but
using the polynomial concentration inequality with moment $2$. We get
here a stronger result since we now have the exponential
concentration inequality at hand.

\begin{proof}
Let
\begin{equation}\label{defKsdf}
K(x_0,\dotsc,x_{n-1})=\sup_{\omega\in[0,2\pi]}\left|
\int_0^\omega \frac{1}{n} \Big| \sum_{j=0}^{n-1} e^{-\ic js} f(x_j)\Big|^2 \dd s -J(\omega)\right|.
\end{equation}
The reader can verify that
\begin{equation}\label{lipsdf}
\sup_{0\leq \ell\leq n-1}\Lip_\ell(K)\leq \frac{c(1+\log n)}{n}
\end{equation}
for some constant $c>0$. Let
\begin{equation}\label{defQ}
Q_n(x)=\sup_{\omega\in[0,2\pi]}\big|J_n(x,\omega)-J(\omega)\big|.
\end{equation}
The major task is to estimate from above $\int Q_n \dd\mu$. We partly
proceed as in \cite[Page 2345]{chazottes_collet_schmitt}: We
discretize $\omega$, that is, given any integer $N\in\N$, we define
the finite sequence of numbers $(\omega_p)$ by $\omega_p=2\pi p/N$,
$p=0,\dotsc,N$. We then define
\[
\overline{Q}_n(x):=\sup_{0\leq p\leq N}\big|J_n(x,\omega_p)-J(\omega_p)\big|.
\]
One can then show that there exists some $C>0$ such that
\begin{equation}\label{QQ}
Q_n(x)\leq \overline{Q}_n(x) + \frac{C}{N}
\end{equation}
for all $x\in X$ and for all integers $n,N\in\N$.

We shall also use the fact (see \cite{chazottes_collet_schmitt} for more details) that
there exists some $C>0$ such that, for all $\omega$ and for any $n\in\N$,
\begin{equation}\label{JJ}
\big|J(\omega)-\int J_n(x, \omega)\dd\mu(x)\big|\leq \frac{C}{n}.
\end{equation}
We now depart from \cite{chazottes_collet_schmitt} and use that for any real $\beta>0$
\begin{equation}\label{pisier}
\int e^{\beta \overline{Q}_n}\dd\mu \leq
\sum_{p=0}^N \int e^{\beta[J_n(x,\omega_p)-J(\omega_p) ]}\dd\mu(x)
+
\sum_{p=0}^N \int e^{\beta[J(\omega_p)-J_n(x,\omega_p)] }\dd\mu(x).
\end{equation}
We estimate each term in the first sum of the right-hand side of this
inequality by using the exponential concentration
inequality~\eqref{main_ineq_bis}, \eqref{lipsdf} and~\eqref{JJ}:
\begin{align*}
\int e^{\beta[J_n(x,\omega_p)-J(\omega_p) ]}\dd\mu(x) & =
\int e^{\beta[J_n(x,\omega_p)-\int J_n(y,\omega_p)\dd\mu(y)]}\dd\mu(x)\cdot e^{\beta[\int J_n(y,\omega_p)\dd\mu(y)-J(\omega_p)]}\\
& \leq e^{C\beta^2 (1+\log n)^2/n}\cdot e^{C\beta/n}.
\end{align*}
We get the same bound for each term in the second sum of the
right-hand side of~\eqref{pisier}, hence
\[
\int e^{\beta \overline{Q}_n}\dd\mu \leq 2(N+1) e^{C\beta^2 (1+\log n)^2/n}\cdot e^{C\beta/n}.
\]
We now use Jensen's inequality,~\eqref{QQ} and~\eqref{defQ} to get
\begin{multline*}
\int \sup_{\omega\in[0,2\pi]}\big|J_n(x,\omega)-J(\omega)\big| \dd\mu(x) \leq \\
\inf_{N\in\N} \left\{\frac{1}{\beta}\log[2(N+1)] + C\beta \frac{(1+\log n)^2}{n}
+\frac{C}{n} + \frac{C}{N}\right\}.
\end{multline*}
It remains to optimize over $N\in\N$ and $\beta>0$ to obtain
\[
\int \sup_{\omega\in[0,2\pi]}\big|J_n(x,\omega)-J(\omega)\big| \dd\mu(x)
\leq \frac{c_1(1+\log n)^{3/2}}{\sqrt{n}}.
\]
We conclude the proof by applying~\eqref{dev_exp} to the
function~\eqref{defKsdf}, taking into account~\eqref{lipsdf} and the
previous estimate.
\end{proof}

\appendix

\section{A technical lemma}

\label{sec_appendix}

Our goal in this section is to prove a technical result that was
required to obtain polynomial concentration estimates in non-uniform
invertible Young towers. Let us consider a non-invertible non-uniform
Young tower in which the return time has a moment of order $q\geq 2$
(i.e., $\sum h^q \mu\{x\in \Delta_0 \st \phi(x)=h\} < \infty$). We
define a function $\Phi_n$ by $\Phi_n(x)=\beta^{\Card\{j\in [1,n]\st
T^j x\in \Delta_0\}}$ for $x\in \Delta_0$, and $\Phi_n=0$ otherwise,
where $\beta<1$ is fixed.

The estimate we need in~\eqref{eq:maineq_invert} is given in the
following theorem.
\begin{thm}
\label{thm_most_technical}
For all nonnegative real numbers $L_k$,
  \begin{equation*}
  \int \left(\sum_r \left(\sum_{k\geq r} L_k \Phi_{k-r}\circ T^r\right)^2\right)^{q-1}
  \leq C \left(\sum L_k^2\right)^{q-1}.
  \end{equation*}
\end{thm}

For the proof, let us expand the square on the left, the resulting
function is bounded by $\sum_r \sum_{k\geq \ell \geq r} L_k L_\ell
\Phi_{k-r}\circ T^r$ since $\Phi_{\ell-r}\circ T^r\leq 1$. Bounding
$L_k L_\ell$ by $L_k^2 + L_\ell^2$, we get two terms that will be
studied separately (but with very similar techniques). The theorem
follows from the following lemmas.
\begin{lem}
\label{lem_first_sum}
We have
  \begin{equation*}
  \int \left(\sum_r \sum_{k\geq r} L^2_k (k-r+1) \Phi_{k-r}\circ T^r\right)^{q-1}
  \leq C \left(\sum L_k^2\right)^{q-1}.
  \end{equation*}
\end{lem}

\begin{lem}
\label{lem_second_sum}
We have
  \begin{equation*}
  \int \left(\sum_r \sum_{k\geq r} \sum_{\ell=r}^{k-1} L^2_\ell  \Phi_{k-r}\circ T^r\right)^{q-1}
  \leq C \left(\sum L_k^2\right)^{q-1}.
  \end{equation*}
\end{lem}

We will prove a more general result, encompassing those two lemmas
and better suited to induction. We will need the following notion.

\begin{definition}
A weight system is a set of numbers $u(r,k)$ for $r<k$ such that
\begin{enumerate}
\item either $u(r,k)=M_k$ for all $r<k$,
\item or $u(r,k)=(\sum_{j=r}^{k-1} M_j)/(k-r)$ for all $r<k$,
\end{enumerate}
where $M_k$ is a summable sequence of nonnegative real numbers. In
both cases, let $\Sigma= \sum M_k$ be the sum of the weight system.
\end{definition}

Weight systems satisfy the following property.
\begin{lem}
\label{lem_somme_uniforme}
Let $u(r,k)$ be a weight system. For all $m>0$, we have $\sum_{r}
u(r,r+m) \leq \Sigma$.
\end{lem}
\begin{proof}
If $u(r,k)=M_k$, then $\sum u(r,r+m) = \sum M_{r+m} \leq \sum
M_r=\Sigma$. If $u(r,k)=(\sum_{j=r}^{k-1}M_j)/(k-r)$, then
  \begin{equation*}
  \sum u(r,r+m) = m^{-1} \sum_r \sum_{j=0}^{m-1} M_{r+j}
  \leq m^{-1} \sum_{j=0}^{m-1} \Sigma = \Sigma.
  \qedhere
  \end{equation*}
\end{proof}

We will also need the following fact.
\begin{lem}
\label{lem_build_new_weights}
Let $u(r,k)$ be a weight system with sum $\Sigma$, and let
$c_n^{(1)}$ be a sequence with a moment of order $1$. There exists a
weight system $v(r,k)$ with sum at most $C\Sigma$ such that, for all
$s<k$, we have $\sum_{r<s} u(r,k) c_{s-r}^{(1)} \leq v(s,k)$.
\end{lem}
\begin{proof}
Let $w(s,k) = \sum_{r<s} u(r,k) c_{s-r}^{(1)}$. If $u(r,k)$ is of the
first type (i.e., $u(r,k)=M_k$), then $w(s,k) = \sum_{r<s} M_k
c_{s-r}^{(1)} \leq CM_k$, and one can take $v(s,k)=CM_k$. If $u(r,k)$
is of the second type (i.e., $u(r,k)=(\sum_{j=r}^{k-1} M_j)/(k-r)$),
then
  \begin{align*}
  w(s,k) & = \sum_{r<s} u(r,k) c_{s-r}^{(1)}
  = \sum_{r<s} \frac{1}{k-r} \left(\sum_{j=r}^{k-1} M_j\right) c_{s-r}^{(1)}
  \\&
  \leq \frac{1}{k-s} \left( \sum_{j<s} M_j \sum_{r\leq j} c_{s-r}^{(1)}
  + \sum_{j=s}^{k-1}M_j \sum_{r<s} c_{s-r}^{(1)}\right)
  \\&
  \leq \frac{1}{k-s} \left( \sum_{j<s} M_j c_{s-j}^{(0)}
  + C\sum_{j=s}^{k-1}M_j \right).
  \end{align*}
Let $M'_s = CM_s + \sum_{j<s} M_j c_{s-j}^{(0)}$, we get $w(s,k) \leq
\frac{1}{k-s} ( M'_s+C\sum_{j=s+1}^{k-1}M_j)$, which is bounded by
$\frac{1}{k-s} \sum_{j=s}^{k-1}M'_j$. Moreover, $\sum M'_j \leq C\sum
M_j$ since the sequence $c_{n}^{(0)}$ is summable. This shows that
$w$ is bounded by a weight system $v$ with sum at most $C\Sigma$.
\end{proof}

The main lemma is the following:
\begin{lem}
Consider a weight system $u(r,k)$, and real numbers $\gamma\geq 1$
and $Q\geq 1$ with $\gamma Q\leq q-1$. We have
  \begin{equation*}
  \int \left( \sum_{k> r}u(r,k) (k-r)^\gamma \Phi_{k-r}\circ T^r \right)^Q
  \leq C \Sigma^{Q}.
  \end{equation*}
\end{lem}
This result implies Lemmas~\ref{lem_first_sum}
and~\ref{lem_second_sum}, using it with $\gamma=1$, $Q=q-1$ and the
weights $L_k^2$ for the former, $(\sum_{\ell=r}^{k-1}
L_\ell^2)/(k-r)$ for the latter.

We will prove the lemma directly for $Q\in [1,2]$, while an induction
will be required for $Q>2$. When $u$ is a weight system, let us write
$S(\gamma,u) = \sum_{k> r}u(r,k) (k-r)^\gamma \Phi_{k-r}\circ T^r$.
We will construct another weight system $v(r,k)$ (with sum at most
$C\Sigma$) such that
  \begin{equation*}
  \int |S(\gamma,u)|^Q \leq C \Sigma^Q + C \Sigma^{Q/2} \int |S(2\gamma, v)|^{Q/2}.
  \end{equation*}
By induction, the last integral is bounded by $C \Sigma^{Q/2}$, and
we obtain the desired result.

Let us explain the strategy of the proof. First, since $\int
\Phi_n\leq c_n^{(q-1)}$ by Lemma~\ref{lem_hatT_Psi} below, we have
  \begin{equation*}
  \E(S(\gamma,u)) \leq \sum_{k>r} (k-r)^\gamma u(r,k) c_{k-r}^{(q-1)}
  = \sum_m m^\gamma c_m^{(q-1)}\left(\sum_r u(r, r+m)\right)
  \leq \sum_m m^\gamma c_m^{(q-1)} \Sigma,
  \end{equation*}
by Lemma~\ref{lem_somme_uniforme}. As $\gamma \leq \gamma Q \leq
q-1$, the sum in $m$ is finite, and we get $\E(S(\gamma,u))\leq
C\Sigma$. Consequently, to prove the lemma, it suffices to bound
$\int |S(\gamma,u)-\E(S(\gamma,u))|^Q$.

We decompose $S=S(\gamma,u)$ as $\E(S)+\sum_{s\geq 0} S_s \circ T^s$,
where $S_s\circ T^s$ is a sequence of reverse martingale differences:
writing $\boF_0$ for the Borel $\sigma$-algebra and
$\boF_s=T^{-s}\boF_0$, the function $S_s\circ T^s$ is
$\boF_s$-measurable and $\E(S_s\circ T^s |\boF_{s+1})=0$, i.e.,
$\E(S_s|\boF_1)=0$. For any function $f$, one has $E(f|\boF_s) =
(\hatT^s f)\circ T^s$, where $\hatT$ is the transfer operator.
Therefore, $S_s$ is given by $S_s(z) = \hatT^s S(z)
-\hatT^{s+1}S(Tz)$.

For $Q\in [1,2]$, the von Bahr-Esseen inequality
\cite{vonbahr_esseen} yields
  \begin{equation*}
  \int |S-\E(S)|^Q \leq \sum_s \E(|S_s|^Q|),
  \end{equation*}
while for $Q>2$ Rosenthal-Burkholder inequality gives an additional
term as follows:
  \begin{equation*}
  \int |S -\E(S)|^Q \leq \E\left( \sum_s \E(S_s^2|\boF_{1})\circ T^s\right)^{Q/2}
  + \sum_s \E(|S_s|^Q).
  \end{equation*}
We will split each function $S_s$ into several parts that will be
estimated separately. Plugging those bounds into the inequalities of
von Bahr-Esseen (for $Q\in  [1,2]$) and Rosenthal-Burkholder (for
$Q>2$) will give the desired result.

More precisely, if $h(x)\not=0$, we have $\E(|S_s| |\boF_1) = 0$ at
the (unique) preimage of $x$ and there is nothing to estimate. On the
other hand, if $h(x)=0$ and if $z$ is a preimage of $x$ under $T$, we
have
  \begin{equation*}
  S_s (z)= \hatT^s S(z) - \hatT^{s+1} S(x)
  = \sum_{k>r} (k-r)^\gamma u(r,k) ( \hatT^s (\Phi_{k-r}\circ T^r)(z) - \hatT^{s+1} (\Phi_{k-r}\circ T^r)(x)).
  \end{equation*}
When estimating $\E(S_s^2 |\boF_1)$ or $\E(|S_s|^Q |\boF_1)$, there
is a contribution coming from $\hatT^{s+1}S(x)$ (involving a sum over
$k>r$), and a contribution coming from the sum over the preimages $z$
of $x$ of $\hatT^s S(z)$ (involving a sum over $z$ and over $k>r$).
We will treat separately those contributions depending on the
positions of $k$ and $r$ with respect to $s$ and to $s-h$ (where $h$
is the height of the preimage $z$ of $x$ one is considering). Let
$\pi z$ be the projection of $z$ in the basis of the tower. If $h\leq
s$, we have $\hatT^s S(z) = \hatT^{s-h} S(\pi z)$. (This is the
interesting case: if $h>s$, then all the following estimates become
easier, we will not indicate the trivial modifications to be done in
this case.)

We will study separately the following cases:
\begin{enumerate}
\item $k>r\geq s+1$, contribution of $\hatT^{s-h}S(\pi
    z)-\hatT^{s+1}S(x)$;
\item $k>s+1> r$, contribution solely of $\hatT^{s+1}S(x)$;
\item $k>s-h$, $\min(s+1,k)>r$, contribution solely of
    $\hatT^{s-h}S(\pi z)$;
\item $s+1\geq k>s-h$, $r<k$, contribution solely of
    $\hatT^{s+1}S(x)$;
\item $s-h\geq k>r$, contribution of $\hatT^{s-h}S(\pi
    z)-\hatT^{s+1}S(x)$.
\end{enumerate}
We will treat separately those five contributions, and see that all
of them satisfy the desired bounds. We will need very precise
estimates on the transfer operator, given in the following lemma. We
recall that the notation $d_n^{(Q)}$ indicates a non-increasing
sequence with a moment of order $Q$.
\begin{lem}
\label{lem_hatT_Psi}
We have $\int \Phi_m \leq c_m^{(q-1)}$. For $h(z)=0$, we have
$\hatT^n \Phi_m(z) \leq c_n^{(q)} \Phi_{m-n}(z)$ if $n \leq m$, and
  \begin{equation*}
  \left|\hatT^n (\hatT^m\Phi_m)(z)
    -\sum_{b\leq n} e(b,m)\right|
  \leq
  \sum_{b=0}^n d_{n-b}^{(q-2)} \sum_{i=0}^m c_{b+m-i}^{(q)}c_i^{(q)},
  \end{equation*}
where the scalar $e(b,m)$ only depends on $b$ and $m$ and is bounded
by $\sum_{i=0}^m c_{b+m-i}^{(q)} c_i^{(q)}$.
\end{lem}
The function $\Phi_m$ involves $m$ iterates of the transformation.
While the transfer operator is eliminating some number $n\leq m$ of
those iterates, the improvement in the estimates depends on $n$, and
$m-n$ iterates remain ready to be used (under the form of
$\Phi_{m-n}$). Once all the variables are eliminated, $\hatT^n
(\hatT^m\Phi_m)$ converges to the integral of $\Phi_m$ (which is
equal to $\sum_{b\geq 0} e(b,m)$), with a more complicated error term
whose precise form will play an important role later on.
\begin{proof}
Let us first assume $n\leq m$. In this case, $\hatT^n
\Phi_m(z)=\Phi_{m-n}(z)\cdot U_n 1(z)$, where the operator $U_n$ was
introduced in the proof of Lemma~\ref{lem_Psin}. We proved there that
$\norm{U_n}\leq c_n^{(q)}$, the desired estimate follows.

For any point $x$ with height $i\in [0,m]$, we obtain $\hatT^m
\Phi_m(x)=\hatT^{m-i}\Phi_m(\pi x) \leq c_{m-i}^{(q)}$. On the other
hand, if $h(x)=i>m$, we have $\hatT^m \Phi_m(x)=\Phi_m(T^{-m}x)=0$,
since $\Phi_m$ vanishes on points with positive height by definition.
Let $\Gamma=\hatT^m \Phi_m$.

We obtain
  \begin{equation*}
  \int \Phi_m = \int \Gamma \leq \sum_{i=0}^m \mu\{h=i\} c_{m-i}^{(q)}
  \leq \sum_{i=0}^m c_i^{(q-1)} c_{m-i}^{(q)} \leq c_m^{(q-1)}.
  \end{equation*}

Let us now study $\hatT^n (\hatT^m \Phi_m)=\hatT^n \Gamma$, using the
previous information regarding $\Gamma$. We will use the operators
$T_k$ and $B_b$ that were introduced in
Subsection~\ref{subsec_renewal}, so that $\hatT^n
\Gamma(z)=\sum_{k+b=n} T_k B_b \Gamma(z)$ for $h(z)=0$. We explained
there that $T_k=\Pi+E_k$ where $\Pi f=(\int f)1_{\Delta_0}$, and
$\norm{E_k}\leq d_k^{(q-2)}$. Hence,
  \begin{equation*}
  \hatT^n \Gamma(z)= \Pi\cdot\sum_{b\leq n} B_b\Gamma
  + \sum_{k+b=n} E_k B_b \Gamma(z).
  \end{equation*}
We estimate first $\norm{B_b\Gamma}$. We have $B_b\Gamma(x)=\sum
g^{(b)}(y) \Gamma(y)$, where we sum over the points $y\in T^{-b}(x)$
not returning to $\Delta_0$ before time $b$. If $h(y)=i$, the point
$\pi y$ has a return time to the basis equal to $b+i$. Therefore,
$|B_b\Gamma(x)|\leq \sum_{i=0}^m c_{b+i}^{(q)} c_{m-i}^{(q)} =
\sum_{i=0}^m c_{b+m-i}^{(q)} c_{i}^{(q)} $ (in view of the bound on
$\Gamma$ at height $i$). The Lipschitz norm of $B_b\Gamma$ is
estimated in the same way. Thus,
  \begin{equation*}
  \sum_{k+b=n} \norm{E_k B_b\Gamma}
  \leq \sum_{k+b=n} d_k^{(q-2)} \sum_{i=0}^m c_{b+m-i}^{(q)}c_i^{(q)}.
  \end{equation*}
Finally, the statement of the lemma is satisfied letting $e(b,m)=\int
B_b\Gamma=\Pi(B_b\Gamma)$. This scalar is independent of $n$ and
bounded by $\sum_{i=0}^m c_{b+m-i}^{(q)} c_i^{(q)}$.
\end{proof}

We will use the following simple remark. For $\kappa\geq 2$ and
$x,y\geq 0$, we have $(x+y)^\kappa \leq x^\kappa + C y
(x+y)^{\kappa-1}$ (by Taylor's formula). By induction, this implies
  \begin{equation}
  \label{ineq_recur}
  \left(\sum_{i=1}^n x_i\right)^\kappa
  \leq C \sum_{i=1}^n x_i \cdot \left(\sum_{j=1}^i x_j \right)^{\kappa-1}.
  \end{equation}

\subsection{The case \texorpdfstring{$k>r\geq s+1$}{k>r>=s+1}}
When $k>r\geq s+1$, we have $\hatT^{s+1}(\Phi_{k-r}\circ
T^r)(x)=\Phi_{k-r}\circ T^{r-s-1}(x)$, while
$\hatT^{s-h}(\Phi_{k-r}\circ T^r)(\pi z) = \Phi_{k-r}\circ
T^{r-s+h}(\pi z)$. Since $T^{h+1}(\pi z)=x$, those terms coincide,
and their contribution to $S_s(z)$ vanishes.

\subsection{The case \texorpdfstring{$k>s+1>r$}{k>s+1>r}, contribution of
\texorpdfstring{$\hatT^{s+1}S(x)$}{L(s+1)S(x)}} The contribution from
$\Phi_{k-r}\circ T^r$ satisfies
  \begin{equation*}
  \hatT^{s+1}(\Phi_{k-r}\circ T^r) = \hatT^{s+1-r} \Phi_{k-r}
  \leq c_{s+1-r}^{(q)} \Phi_{k-s-1}(x),
  \end{equation*}
by Lemma~\ref{lem_hatT_Psi}. Summing those contributions to $S_s(z)$
(for varying $k$ and $r$) gives a term which is bounded by
  \begin{equation*}
  S_s^{(2)} = \sum_{k>s+1>r} (k-r)^\gamma u(r,k) c_{s+1-r}^{(q)} \Phi_{k-s-1}(x).
  \end{equation*}
Let us note that this term does not depend on $z$. Since $k-r =
(k-s-1)+(s+1-r) \leq 2 (k-s-1)(s+1-r)$ and since $(s+1-r)^\gamma
c_{s+1-r}^{(q)} \leq c_{s+1-r}^{(q-\gamma)}$, we have
  \begin{equation*}
  S_s^{(2)} \leq \sum_{k>s+1} \sum_{r\leq s} u(r,k) c_{s+1-r}^{(q-\gamma)}
  (k-s-1)^\gamma \Phi_{k-s-1}(x).
  \end{equation*}
By Lemma~\ref{lem_build_new_weights}, there exists a new weight
system $v$ such that $\sum_{r\leq s} u(r,k) c_{s+1-r}^{(q-\gamma)}
\leq v(s+1,k)$, yielding $S_s^{(2)} \leq \sum_{k>s+1} v(s+1,k)
(k-s-1)^\gamma \Phi_{k-s-1}(x)$. Moreover, the sum of the weight $v$
is at most $C\Sigma$.

Let $\kappa\geq 1$, we estimate $|S_s^{(2)}(z)|^\kappa$. We apply the
inequality~\eqref{ineq_recur} to $x_k = v(s+1,k) (k-s-1)^\gamma
\Phi_{k-s-1}$, yielding
  \begin{equation*}
  |S_s^{(2)}|^\kappa \leq \sum_{k>s+1} v(s+1,k) (k-s-1)^\gamma \Phi_{k-s-1} \cdot
  \left( \sum_{s+1< \ell \leq k} v(s+1,\ell) (\ell-s-1)^\gamma \right)^{\kappa-1}.
  \end{equation*}
We claim that the last sum is bounded by $C(k-s-1)^\gamma \Sigma$.
Indeed, if the weight $v$ is of the first type (i.e.,
$v(r,\ell)=M_\ell$), then we bound $(\ell-s-1)^\gamma$ by
$(k-s-1)^\gamma$, to obtain $(k-s-1)^\gamma\sum_{\ell=s+2}^k M_\ell
\leq C(k-s-1)^\gamma \Sigma$. On the other hand, if $v$ is of the
second type (i.e., $v(r,\ell)=(\sum_{j=r}^{\ell-1} M_j)/(\ell-r)$),
then the sum is bounded by
  \begin{align*}
  \sum_{\ell=s+2}^k \sum_{j=s+1}^{\ell-1} M_j (\ell-s-1)^{\gamma-1}
  &\leq (k-s-1)^{\gamma-1} \sum_{j=s+1}^{k-1} M_j (k-j)
  \\&
  \leq (k-s-1)^\gamma \sum_{j=s+1}^{k-1} M_j \leq (k-s-1)^\gamma \Sigma.
  \end{align*}
We have proved that, for all $\kappa\geq 1$,
  \begin{equation}
  \label{eq_domine_Sb2}
  |S_s^{(2)}|^\kappa \leq C\sum_{k>s+1} v(s+1,k) (k-s-1)^{\kappa\gamma} \Phi_{k-s-1}  \Sigma^{\kappa-1}.
  \end{equation}

Let us now assume that $Q\in [1,2]$, and let us consider the
contribution of $S_s^{(2)}$ to von Bahr-Esseen inequality. It is
given by
  \begin{equation*}
  \sum_s \E( |S_s^{(2)}|^Q) =
  \sum_s \E( \E(|S_s^{(2)}|^Q|\boF_{1}))
  \leq \sum_s C\sum_{k>s+1} v(s+1,k) (k-s-1)^{Q\gamma} \E(\Phi_{k-s-1})  \Sigma^{Q-1},
  \end{equation*}
by~\eqref{eq_domine_Sb2}. Since $\E(\Phi_{k-s-1}) \leq
c^{(q-1)}_{k-s-1}$, this can be written (letting $k=s+1+m$) as
$\Sigma^{Q-1}\sum_m m^{Q\gamma}c_m^{(q-1)}\sum_s v(s+1, s+1+m)$. For
fixed $m$, the sum $\sum_s v(s+1, s+1+m)$ is bounded by $C\Sigma$ by
Lemma~\ref{lem_somme_uniforme}. As $Q\gamma \leq q-1$,
$m^{Q\gamma}c_m^{(q-1)}$ is summable, and we obtain a bound
$C\Sigma^Q$ as desired.

Assume now $Q>2$. In this case, the second term in the
Rosenthal-Burkholder inequality is bounded by $C\Sigma^Q$ as above.
Using~\eqref{eq_domine_Sb2} (with $\kappa=2$), the first term is at
most
  \begin{equation*}
  C\int \left( \sum_s \sum_{k>s+1} v(s+1,k) (k-s-1)^{2\gamma} \Phi_{k-s-1}\circ T^{s+1}\cdot \Sigma \right)^{Q/2}
  = C\Sigma^{Q/2} \int |S(2\gamma, v)|^{Q/2}.
  \end{equation*}
Since $\gamma'=2\gamma$ and $Q'=Q/2$ satisfy $\gamma'Q'\leq q-1$, we
can argue by induction to show that this term is again bounded by
$\Sigma^Q$.

\subsection{The case \texorpdfstring{$k>s-h$, $\min(s+1,k)> r$}{k>s-h, min(s+1, k)>r}, contribution of
    \texorpdfstring{$\hatT^{s-h}S(\pi z)$}{L(s-h)S(pi z)}}
We should study $S^{(3)}_s(z) = \hatT^{s-h}(\sum_{k>s-h} \sum_{r\leq
\min(s, k-1)} u(r,k) (k-r)^\gamma \Phi_{k-r}\circ T^r)(\pi z)$.

If $k>s-h$ and $r\in (s-h, s]$ with $r<k$, we have
$\hatT^{s-h}(\Phi_{k-r}\circ T^r)(\pi z) = \Phi_{k-r}\circ
T^{r-(s-h)}(\pi z)$. Since the point $T^{r-(s-h)}(\pi z)$ has
positive height, the function $\Phi_{k-r}$ vanishes here. Therefore,
we only have to consider the contribution of $k>s-h\geq r$. This is
exactly the same thing as in the previous subsection, but for the
point $\pi z$ instead of $x$. The inequality~\eqref{eq_domine_Sb2}
gives, for all $\kappa\geq 1$,
  \begin{equation*}
  |S_s^{(3)}(z)|^\kappa \leq C\sum_{k>s-h} v(s-h,k) (k-s+h)^{\kappa\gamma} \Phi_{k-s+h}(\pi z)  \Sigma^{\kappa-1},
  \end{equation*}
where $v$ is a weight system with sum at most $C\Sigma$. For $k\in
(s-h,s+1]$, we simply bound $\Phi_{k-s+h}(\pi z)$ by $1$, while for
$k>s+1$ we bound it by $\Phi_{k-s-1}(x)$, since $T^{h+1}(\pi z)=x$.
Summing over the preimages $z$ of $x$, we get
  \begin{align*}
  \E(|S_s^{(3)}|^\kappa|\boF_1)
  \leq C\Sigma^{\kappa-1}\sum_{h\geq 0} c_h^{(q)} \Biggl(
  &\sum_{k=s-h+1}^{s+1} v(s-h,k) (k-s+h)^{\kappa\gamma}
  \\ &
  + \sum_{k>s+1} v(s-h,k) (k-s+h)^{\kappa\gamma} \Phi_{k-s-1}(x)\Biggr).
  \end{align*}
In the first sum, we bound $k-s+h$ by $h+1$ and we use the inequality
$(h+1)^{\kappa\gamma} c_h^{(q)} \leq c_h^{(q-\kappa\gamma)}$. In the
second sum, we have $c_h^{(q)} (k-s+h)^{\kappa\gamma} \leq
c_h^{(q-\kappa\gamma)} (k-s-1)^{\kappa\gamma}$ by the same argument.
If $\kappa\gamma\leq q-1$, the quantity $\sum_{h\geq 0}
c_h^{(q-\kappa\gamma)} v(s-h, k)$ is bounded by $w(s+1,k)$ where $w$
is a weight system with sum at most $C\Sigma$, by
Lemma~\ref{lem_build_new_weights}. We obtain
  \begin{equation}
  \label{wmpxcvu}
  \begin{split}
  \E(|S_s^{(3)}|^\kappa|\boF_1)
  \leq C\Sigma^{\kappa-1} \Biggl(
  &
  \sum_{h\geq 0} \sum_{k=s-h+1}^{s+1} c_h^{(q-\kappa\gamma)}v(s-h,k)
  \\ &
  + \sum_{k>s+1} w(s+1,k) (k-s-1)^{\kappa\gamma} \Phi_{k-s-1}(x)\Biggr).
  \end{split}
  \end{equation}
The second term is identical to the term appearing in the previous
subsection, in~\eqref{eq_domine_Sb2}. It follows in the same way that
its contribution to the inequalities of von Bahr-Esseen (case $Q\in
[1,2]$) and Rosenthal-Burkholder (case $Q>2$) is bounded by
$C\Sigma^Q$.

Let us consider the first term, first in von Bahr-Esseen inequality
(case $Q\in [1,2]$). Thanks to~\eqref{wmpxcvu} (with $\kappa=Q$), its
contribution is given by
  \begin{align*}
  \sum_s C\Sigma^{Q-1} \sum_{h\geq 0} \sum_{k=s-h+1}^{s+1} c_h^{(q-Q\gamma)}v(s-h,k)
  &= C\Sigma^{Q-1} \sum_{h\geq 0} c_h^{(q-Q\gamma)} \sum_{m=1}^{h+1} \sum_s v(s-h, s-h+m)
  \\&\leq C \Sigma^{Q-1} \sum_{h\geq 0} c_h^{(q-Q\gamma)} \sum_{m=1}^{h+1} \Sigma
  = C \Sigma^{Q} \sum_{h\geq 0} c_h^{(q-Q\gamma-1)},
  \end{align*}
where we used Lemma~\ref{lem_somme_uniforme} for the inequality.
Since $Q\gamma\leq q-1$, this is bounded by $C\Sigma^Q$.

When $Q>2$, we use the Rosenthal-Burkholder inequality. As above, the
last term in this inequality is bounded by $C\Sigma^Q$.
Using~\eqref{wmpxcvu} (with $\kappa=2$), the first term is bounded by
  \begin{equation*}
  \left(\sum_s  C \Sigma  \sum_{h\geq 0} \sum_{k=s-h+1}^{s+1} c_h^{(q-2\gamma)}v(s-h,k)\right)^{Q/2}.
  \end{equation*}
The same computation as above shows that this is bounded by
$(C\Sigma^2)^{Q/2}$.

\subsection{The case \texorpdfstring{$s+1\geq k>s-h$, $r<k$}{s+1>=k>s-h, r<k}, contribution of
    \texorpdfstring{$\hatT^{s+1}S(x)$}{L(s+1)S(x)}}
The contribution coming from $\Phi_{k-r}\circ T^r$ satisfies
  \begin{equation*}
  \hatT^{s+1}(\Phi_{k-r}\circ T^r) = \hatT^{s+1-k} \hatT^{k-r} \Phi_{k-r},
  \end{equation*}
which is controlled by Lemma~\ref{lem_hatT_Psi}. Summing over $k\in
[s-h+1, s+1]$ and $r<k$, we obtain that the resulting contribution
$S_s^{(4)}$ is bounded by
  \begin{align*}
  \sum_{k=s-h+1}^{s+1} \sum_{r<k}u(r,k) (k-r)^\gamma \Biggl( &
  \sum_{b\leq s+1-k} \sum_{i=0}^{k-r} c_{b+k-r-i}^{(q)}c_i^{(q)}
  \\&
  + \sum_{b\leq s+1-k} d_{s+1-k-b}^{(q-2)} \sum_{i=0}^{k-r} c_{b+k-r-i}^{(q)}c_i^{(q)}\Biggr).
  \end{align*}
Since $d_{s+1-k-b}^{(q-2)}$ is bounded, the second term is bounded by
the first one. Since $k-r\leq (b+k-r-i) + i$, we have $k-r \leq
(b+k-r-i+1)(i+1)$, yielding $(k-r)^\gamma c_{b+k-r-i}^{(q)}c_i^{(q)}
\leq c_{b+k-r-i}^{(q-\gamma)}c_i^{(q-\gamma)}$. For $\kappa\geq 1$,
we obtain (letting $m=k-r$)
  \begin{equation*}
  \E(|S_s^{(4)}|^\kappa|\boF_{1})
  \leq \sum_{h\geq 0} c_h^{(q)} \left( \sum_{k=s-h+1}^{s+1} \sum_{b\leq s+1-k}
  \sum_{i\geq 0} c_i^{(q-\gamma)} \sum_{m\geq i} u(k-m, k) c_{b+m-i}^{(q-\gamma)}\right)^\kappa.
  \end{equation*}
Summing over $s$ and using the inequality $\sum x_i^\kappa \leq (\sum
x_i)^\kappa$, we get
  \begin{equation*}
  \sum_s \E(|S_s^{(4)}|^\kappa|\boF_{1})\circ T^s
  \leq \sum_{h\geq 0} c_h^{(q)} \left( \sum_s \sum_{k=s-h+1}^{s+1} \sum_{b\leq s+1-k}
  \sum_{i\geq 0} c_i^{(q-\gamma)} \sum_{m\geq i} u(k-m, k) c_{b+m-i}^{(q-\gamma)}\right)^\kappa.
  \end{equation*}
We reorganize the sums as follows. First, we write $s+1 = k+a$ for
some $a\in [0, h]$, so that the first three sums are replaced by
$\sum_{a=0}^h \sum_k \sum_{b\leq a}$. Then, we move the sum over $k$
to the end: since $\sum_k u(k-m, k) \leq \Sigma$ for all $m$ by
Lemma~\ref{lem_somme_uniforme}, we get a bound
  \begin{equation*}
  \Sigma^\kappa \sum_{h\geq 0} c_h^{(q)} \left(\sum_{a=0}^h \sum_{b\leq a} \sum_{i\geq 0} c_i^{(q-\gamma)}
  \sum_{m\geq i} c_{b+m-i}^{(q-\gamma)}\right)^\kappa.
  \end{equation*}
The sum over $m\geq i$ is bounded by $d_b^{(q-\gamma-1)}$. The
(finite) quantity $\sum_{i\geq 0} c_i^{(q-\gamma)}$ can be factorized
out, giving a multiplicative constant. Since the sum $\sum_{b\leq a}
d_b^{(q-\gamma-1)}$ is uniformly bounded, we get an upper bound
$\Sigma^\kappa \sum_{h\geq 0} (h+1)^\kappa c_h^{(q)} \leq
C\Sigma^\kappa$, when $\kappa\leq q$.

This readily implies that the contributions of $S_s^{(4)}$ to the
inequalities of von Bahr-Esseen (case $1\leq Q\leq 2$) and
Rosenthal-Burkholder (case $Q>2$) are bounded by $\Sigma^Q$, as
desired.

\subsection{The case \texorpdfstring{$s-h\geq k>r$}{s-h>=k>r}}
The contribution coming from $\Phi_{k-r}\circ T^r$ reads
  \begin{equation*}
  \hatT^{s-h}(\Phi_{k-r}\circ T^r)(\pi z) - \hatT^{s+1}(\Phi_{k-r}\circ
  T^r)(x)
  =\hatT^{s-h-k} \hatT^{k-r}\Phi_{k-r}(\pi z) - \hatT^{s+1-k} \hatT^{k-r}\Phi_{k-r}(x).
  \end{equation*}
To estimate those contributions, we use Lemma~\ref{lem_hatT_Psi}. The
main terms $e(b, k-r)$ simplify partially: only those corresponding
to $s-h-k<b \leq s+1-k$ remain. As a consequence, the global
contribution $S_s^{(5)}(z)$ is bounded by
  \begin{equation*}
  \sum_{s-h\geq k>r} (k-r)^\gamma u(r,k)\left(
  \sum_{b=s-h-k+1}^{s+1-k} \sum_{i=0}^{k-r}c_{b+k-r-i}^{(q)}c_i^{(q)}
  + \sum_{b=0}^{s-h-k} d_{s-h-k-b}^{(q-2)} \sum_{i=0}^{k-r}c_{b+k-r-i}^{(q)}c_i^{(q)}\right).
  \end{equation*}
Let us first note that $(k-r)^\gamma c_{b+k-r-i}^{(q)}c_i^{(q)} \leq
c_{b+k-r-i}^{(q-\gamma)}c_i^{(q-\gamma)}$ as in the previous
subsection. We will then handle separately the two pieces
$S_s^{(5.1)}(z)$ and $S_s^{(5.2)}(z)$ of this expression.

Summing over $h$ and then over $s$, and using the inequality $\sum
x_i^\kappa \leq (\sum x_i)^\kappa$ as in the previous subsection, we
get
  \begin{equation*}
  \sum_s \E(|S_s^{(5.1)}|^\kappa|\boF_{1})\circ T^s
  \leq \sum_{h\geq 0} c_h^{(q)} \left( \sum_s \sum_{k\leq s-h} \sum_{b=s-h-k+1}^{s+1-k}
  \sum_{i\geq 0} c_i^{(q-\gamma)} \sum_{m\geq i} u(k-m, k) c_{b+m-i}^{(q-\gamma)}\right)^\kappa.
  \end{equation*}
Let us reorganize the sums essentially as in the previous subsection.
First, let $s+1-h=k+a$ for some $a\geq 1$, so that the first sums
become $\sum_{a\geq 1} \sum_k \sum_{b=a}^{a+h}$. Then, we move the
sum over $k$ to the end, and we use the inequality $\sum_k
u(k-m,k)\leq \Sigma$ for all $m$. This yields a bound
  \begin{equation*}
  \Sigma^\kappa \sum_{h\geq 0} c_h^{(q)} \left( \sum_{a\geq 1} \sum_{b=a}^{a+h}
  \sum_{i\geq 0} c_i^{(q-\gamma)} \sum_{m\geq i} c_{b+m-i}^{(q-\gamma)}\right)^\kappa.
  \end{equation*}
The last sum over $m$ is bounded by $d_b^{(q-\gamma-1)}$, which is
independent of $i$. Therefore, we may factorize out the sum over $i$,
since $\sum_i c_i^{(q-\gamma)} < \infty$. Since $d_b^{(q-\gamma-1)}$
is nonincreasing, we have $\sum_{b=a}^{a+h} d_b^{(q-\gamma-1)} \leq
(h+1) d_a^{(q-\gamma-1)}$. As $q-\gamma-1\geq 0$, the sequence
$d_a^{(q-\gamma-1)}$ is summable, giving yet another multiplicative
constant. We obtain a bound $C \Sigma^\kappa \sum_{h\geq 0}
(h+1)^\kappa c_h^{(q)} \leq C\Sigma^\kappa$ when $\kappa\leq q$.

Let us now study $S_s^{(5.2)}(z)$. We have
  \begin{multline*}
  \sum_s \E(|S_s^{(5.2)}|^\kappa|\boF_{1})\circ T^s
  \\
  \leq \sum_{h\geq 0} c_h^{(q)} \left( \sum_s \sum_{k\leq s-h} \sum_{b=0}^{s-h-k}
  d_{s-h-k-b}^{(q-2)}
  \sum_{i\geq 0} c_i^{(q-\gamma)} \sum_{m\geq i} u(k-m, k) c_{b+m-i}^{(q-\gamma)}\right)^\kappa.
  \end{multline*}
We proceed exactly as above, with the difference that the sum over
$b$ goes from $0$ to $a-1$. We get a bound
  \begin{equation*}
  C\Sigma^\kappa \sum_{h\geq 0} c_h^{(q)} \left(\sum_{a\geq 1} \sum_{b=0}^{a-1} d_{a-1-b}^{(q-2)}
  \cdot d_b^{(q-\gamma-1)} \right)^\kappa.
  \end{equation*}
Since $q-\gamma-1 \leq q-2$, the convolution between
$d_{a-1-b}^{(q-2)}$ and $d_b^{(q-\gamma-1)}$ is bounded by
$c_{a-1}^{(q-\gamma-1)}$. As $\gamma+1\leq q$, the sum over $a$ is
finite, and we obtain a bound $\Sigma^\kappa$.

Gluing the two pieces together, we have shown that $\sum_s
\E(|S_s^{(5)}|^\kappa |\boF_{1})\circ T^s \leq C\Sigma^\kappa$ for
all $\kappa\leq q$. This readily implies that the contributions of
$S_s^{(5)}$ to the inequalities of von Bahr-Esseen (case $1\leq Q\leq
2$) and Rosenthal-Burkholder (case $Q>2$) are bounded by $\Sigma^Q$,
as desired.

\bibliography{biblio}
\bibliographystyle{amsalpha}
\end{document}